\documentclass{article}
\usepackage[utf8]{inputenc}

\usepackage{geometry}
\usepackage{eufrak}
\usepackage{cancel}
\usepackage{dsfont}

\usepackage{appendix}
\usepackage{amsfonts, amssymb, amsmath, stmaryrd, amsthm, accents, bbm}
\usepackage{mathtools}
\usepackage[linesnumbered,ruled,vlined]{algorithm2e}

\usepackage{transparent}
\usepackage{caption2}

\usepackage[pdftex,
colorlinks=true, 
linkcolor=blue, 
citecolor=blue, 
pagebackref=false, 
]{hyperref}

\usepackage{makeidx}
\makeindex

\usepackage[version=4,arrows=pgf-filled,
textfontname=sffamily,
mathfontname=mathsf]{mhchem}
\usepackage{braket}
\usepackage{comment}
\usepackage[normalem]{ulem}

\usepackage{todonotes}
\newtheorem{proposition}{Proposition}[section]
\newtheorem{corollary}{Corollary}[section]
\newtheorem{definition}{Definition}[section]
\newtheorem{theorem}{Theorem}[section]
\newtheorem{lemma}{Lemma}[section]
\newtheorem{remark}{Remark}[section]
\newtheorem{assumption}{Assumption}[section]
\newtheorem{example}[theorem]{Example}

\newcommand{\cU}{\mathcal{U}}

\DeclareMathOperator{\Tr}{{\text{Tr}}}

\newcommand{\norm}[1]{\left\Vert #1 \right\Vert}
\newcommand{\abs}[1]{\left\vert #1 \right\vert}
\newcommand{\ind}{\mathds{1}}
\newcommand{\R}{\mathbb{R}}

\newcommand{\E}{\mathbb{E}}
\newcommand{\F}{\mathbb{F}}
\newcommand{\bP}{\mathbb{P}}
\newcommand{\Q}{\mathbb{Q}}

\newcommand{\cF}{\mathcal{F}}
\newcommand{\cL}{\mathcal{L}}

\newcommand{\cH}{\mathcal{H}}
\newcommand{\cA}{\mathcal{A}}

\DeclareMathOperator{\diag}{diag}
\DeclareMathOperator*{\argmax}{arg\,max}

\DeclareMathOperator{\diver}{div}

\DeclareMathOperator*{\esssup}{ess\,sup}

\DeclareMathOperator*{\trace}{Tr}

\def\dd{\mathrm{d}}

\def\argmax{\mathop{\rm argmax}}

\title{Growth model with externalities for energetic transition via MFG with common external variable} 

\author{Pierre Lavigne\footnote{Université C{\^o}te d'Azur. E-mail: \href{pierre.lavigne@unice.fr}{pierre.lavigne@unice.fr}.}, Quentin Petit\footnote{EDF R\&D and FiME Lab. E-mail: \href{quentin.petit@edf.fr}{quentin.petit@edf.fr}}, Xavier Warin\footnote{EDF R\&D and FiME Lab. E-mail: \href{xavier.warin@edf.fr}{xavier.warin@edf.fr}}}

\begin{document}
\maketitle

\begin{abstract}
    This article introduces a novel mean-field game model for multi-sector economic growth in which a dynamically evolving externality, influenced by the collective actions of countries, plays a central role. Building on classical growth theories and integrating environmental considerations, the framework incorporates “common noise” to capture shared uncertainties among countries about the externality variable. We establish the existence and uniqueness of the mean-field game equilibrium by reformulating the equilibrium conditions as a Forward–Backward Stochastic Differential Equation via the stochastic maximum principle, first applying a contraction-mapping argument to guarantee a unique solution, then employing the concept of weak equilibria to prove existence under more general assumptions, and finally invoking a specific monotonicity regime to reaffirm uniqueness. We provide a numerical resolution for a specified model using a fixed-point approach combined with neural network approximations.
\end{abstract}


\section{Introduction}

Mean-field game (MFG) theory, first introduced by Lasry and Lions \cite{lasry2007mean} and independently by Huang, Malhamé, and Caines \cite{huang2006large}, has emerged as a powerful framework for modelling interactions within large populations of agents. In these models, each agent optimises its strategy based on the aggregate behaviour of the population, leading to Nash equilibria, which are easier to analyse in the infinite-agent limit. MFGs have found applications in economics, finance, and environmental modelling, providing essential tools for studying distributed decision-making in complex systems.

The incorporation of common noise, random external factors that affect all agents simultaneously, into MFGs has been an area of growing interest. This extension introduces additional complexity but also broadens the applicability of MFGs to real-world scenarios where agents are subject to shared uncertainties. Foundational work in this area includes studies by Ahuja \cite{ahuja2016wellposedness},  Cardaliaguet, Delarue, Lasry, and Lions \cite{cardaliaguet2019master}, Carmona and Delarue \cite[Vol. II]{carmona2018probabilistic} and Carmona, Delarue, and Lacker \cite{carmona2016mean}. These contributions address the well-posedness of MFGs with common noise and provide insights into the convergence of finite-agent systems to their mean-field counterparts. Recent research by Djete \cite{djete2023large} further explores the impact of common noise on interactions through controls.

In this paper, we propose a novel mean-field game model in which an infinite number of countries (the agents here) interact via a dynamically evolving common externality in a multi-sector economic growth framework. Because of the mean field regime, the countries are assumed to be atomeless. That is to say, a single country cannot influence the externality variable. To fix the ideas, the externality is understood as being the \ce{CO2} concentration in the atmosphere, but the dynamics of the externality variable is general enough to capture other kind of externalities (\ce{CH4} concentration, mean surface temperature, level of resource stocks, biodiversity loss, etc).

This externality, driven by aggregate countries actions, introduces a structure that differs from classical MFG formulations, leading to new challenges in both analysis and computation. Notably, unlike the master equation typically associated with games involving common noise, which generally includes a second-order derivative with respect to the measure, the master equation in our framework involves only first-order derivatives with respect to the measure as noted in Section \ref{subsec:master_equation}. This structural distinction aligns with observations in Bertucci's work on monotone solutions for MFG master equations \cite{bertucci2023monotone}. Our setting also resonates with work on trading by Cardaliaguet and Lehalle \cite{Cardaliaguet2018} and recent developments in mean-field game theory, such as the incorporation of noise through an additional variable in finite state spaces by Bertucci and Meynard \cite{bertucci2024noise}. Furthermore, the use of advanced techniques such as Malliavin calculus to handle MFGs with common noise, as explored by Tangpi and Wang \cite{tangpi2025}, underscores the increasing complexity and versatility of these models. Our framework contributes to this growing body of work by offering a novel perspective on multi-sector economic growth under the influence of a common externality.

The proposed model builds upon classical growth theory, such as the two-sector models developed by Uzawa \cite{uzawa1961two}, and extends it to incorporate environmental and sustainability considerations. Research on green growth models by Smulders, Toman, and Withagen \cite{smulders2014growth}, Tahvonen and Kuuluvainen \cite{tahvonen1993economic}, and the Green Solow Model by Brock and Taylor \cite{brock2010green} highlights the importance of integrating environmental externalities into growth dynamics. These approaches motivate our investigation of how countries' investment decisions in different types of capital affect common resources or external factors.

Recent advances in mean-field games have highlighted their potential to address growth and distributional dynamics in more complex settings. For example, Achdou et al. \cite{achdou2023mean} use a mean-field framework to study interactions between firms in input markets, providing insights into competitive equilibria in complex economic networks. Similarly, Achdou et al. \cite{achdou2022income} study wealth and income distributions in macroeconomic contexts, providing continuous-time approaches to understanding inequality. Gomes, Lafleche, and Nurbekyan \cite{gomes2016mean} extend mean-field game theory to model economic growth, addressing investment dynamics and sectoral interactions. In addition, Zhou and Huang \cite{zhou2024best} study stochastic growth games with common noise, illustrating the impact of common uncertainties on optimal strategies and equilibrium states. Taken together, these papers highlight the versatility of mean-field game models for capturing the interaction between individual decisions and aggregate economic outcomes.

This article is situated within a broader context of applications, including environmental economics and sustainable development. For instance, MFG models have been used to study the tragedy of the commons by Kobeissi, Mazari-Fouquer, and Ruiz-Balet \cite{kobeissi2024tragedy}, to analyse the decarbonisation of financial markets by Lavigne and Tankov \cite{lavigne2023decarbonization}, and to investigate pollution regulation under a cap-and-trade system by Del Sarto, Leocata, and Livieri \cite{Delsarto2024}.

\paragraph{Contributions.}

We contribute to the Integrated Assessment Models (IAMs) literature. IAMs are economic models of development, incorporating couplings with respect to the resources and environmental variables, see \cite{bilal2025macroeconomics} and \cite{weyant2017some} for surveys. The development of IAMs started with the seminal work of \cite{nordhaus1993rolling} where the author  incorporate a climate module into a Solow type growth model. Later, the DICE model introduced in \cite{nordhaus1993rolling} was extended in \cite{nordhaus1996regional} to account for strategic interactions among countries through the climate system and international trade within a finite-player game framework.
In our framework, we adopt a similar modeling philosophy but move to a mean-field setting: we consider an infinite number of atomless countries interacting through a common environmental externality. In addition, we embed this interaction in a multi-sector growth structure, which allows us to capture sectoral heterogeneity.
IAMs models aim at better understanding the nature of the coupling between climate and economics to recommand policies preventing and mitigating the effects of climate change on societies. For practical reasons, they often simplifies the overall economical system as one representative country.

Climate change is a perfect illustration of the tragedy of commons \cite{hardin1968tragedy}. Climate can be seen as a common good shared by a finite number of interacting countries deciding on their amount of effort to preserve its quality. Because the common good is shared, incentives to preserve its quality is reduced, and its effective conservation is determined as the output of a Nash equilibrium problem.
On the one hand, taking into account for the interaction among every country is impossible for computational reasons. On the other hand, strategic behaviours is a key aspect of the problematic, which might lead to fundamental misconception and irrelevant  solutions if not taken into account. 
We fill this gap, proposing a limit model where the strategic behaviour of each country is taken into account, expected to approximate the finite player case, while being computationally tractable.   However, a rigorous derivation of this model as the limit of a $N$-players game when the number of players tends to infinity is not addressed here and would require additional mathematical developments. Such a justification constitutes an interesting research direction in its own right \cite{cardaliaguet2019master}.

From a modelling perspective, we introduce, to the best of our knowledge, the first multi-sector economic growth model of mean-field game type in which interactions between agents occur through a common external variable. While this modelling feature seems natural, it raises new mathematical challenges, see \cite{bertucci2023monotone,bertucci2024noise,meynard2024study} for recent contributions. This modelling characteristic has also been considered in a mathematical finance context \cite{Cardaliaguet2018}, but the associated mathematical difficulties vanishes due to structural assumptions on the cost functional of the countries. See \cite[Vol. 1, Remark 1.20]{carmona2018probabilistic} for a discussion. Our framework is closely related to the recent study \cite{aid2025regulation}, wherein they explored an 
N-agents growth model and its mean field limit. However, our study diverges in that we incorporate common noise into the external variable dynamics, rendering the application of their shooting method-based Nash equilibrium existence and uniqueness result unattainable.
The presence of common noise negatively affects well-posedness. In the contraction regime, it reduces the size of the contraction constant, while in the monotone regime, it slightly diminishes the admissible monotone region, as only one term in the monotonicity condition depends on the common noise.

From a mathematical perspective, we address both the theoretical and computational challenges posed by our model. We demonstrate the existence and uniqueness of a strong mean-field game equilibrium. Using the stochastic maximum principle, we reformulate the equilibrium conditions as a Forward-Backward Stochastic Differential Equation (FBSDE). We establish both the existence and the uniqueness of equilibria under successive sets of hypotheses. Using a contraction-mapping argument, as in \cite[Chapter 1, Vol. II]{carmona2018probabilistic}, we derive a first condition guaranteeing a unique fixed point. We then invoke the concept of weak equilibria, following \cite[Vol. II]{carmona2018probabilistic}, to prove existence under much more general assumptions. Finally, following the approach of \cite{bertucci2024noise, meynard2024study}, we show that the equilibrium is again unique thanks to a displacement monotonicity condition \cite{carmona2018probabilistic,Gangbo2022a, Gangbo2022b, Meszaros2024}.


On the computational side, inspired by Carmona and Laurière \cite{carmona2022convergence}, we develop a fixed-point algorithm combined with neural network approximations to solve the resulting optimisation problems. At each iteration, we fix the common contribution to the externality and approximate the optimal control using a neural network. The algorithm alternates between solving the optimisation problem faced by countries and updating the estimate of the aggregate contribution to externality. The main difference with the recently cited work is that a second neural network, and thus a second optimisation, is used to estimate the common contribution at each iteration. 
To stabilize convergence, particularly for large time horizons, we incorporate a fictitious time-stepping technique
(see for example \cite{cardaliaguet2017learning}). When the classical fixed-point iteration does not clearly converge, 
this technique helps the scheme reach convergence. It allows us to postpone the non-convergence of the fixed-point 
iteration scheme as the maturity increases.

   \paragraph{Policy Implication.} 
   Outputs of Integrated Assessment Models (IAMs), such as the social cost of carbon, are key quantities for informing policy decisions at both the national level and in international negotiations. Some IAMs, particularly those inspired by the RICE framework, attempt to incorporate strategic interactions among countries, but face a curse of dimensionality: as the number of countries increases, the problem quickly becomes intractable. In this work, we show that it is possible to consider a limiting regime with infinitely many countries that remains numerically tractable and is even theoretically well-posed under certain conditions, paving the way for a new generation of models.

\paragraph{Structure of the paper.}
The rest of the paper is structured as follows. Section \ref{sec:not-toolbox} introduces the notations used throughout the analysis. In Section \ref{sec:tech-inv-model}, we describe the technological investment model in detail.  In Section \ref{sec:math-analysis}, we present the theoretical analysis, including key assumptions, the stochastic maximum principle, and the proof of the existence and uniqueness of the equilibrium. Section \ref{sec:num-sim} details our numerical approach and provides simulation results for a specified model.

\section{Notations} \label{sec:not-toolbox}

In this section, we provide the main notations of the article.
We denote by $C^{n}$ the set of functions with $n$ continuous derivatives, and by $C^{1,1}$ the set of functions belonging to $C^1$ with Lipschitz derivative. For a function $f$, the notation $\nabla f$ denotes its gradient. 
If $f$ is a Lipschitz function of the form $(x,y) \mapsto f(x,y)$ we denote by $C_{f,x}$ its Lipschitz constant with respect to the variable $x$. Given a vector $x$ in $\R^n$, we denote by $\abs{x}$ the $2$-norm of $x$. Given a matrix $M \in \mathbb{R}^{n\times n}$, we use the notation $\diag(M)$ to denote a vector of size $n$ containing the diagonal elements of the matrix $M$, that is to say
\begin{equation*}
    \diag(M) = (M_{i,i})_{0\leq i \leq n}.
\end{equation*}
\paragraph{Spaces of random variables and random processes.} 

\noindent \textit{$L^2(\mathcal G,\mathbb{R}^d)$ spaces.} 
For a given $\sigma$-field $\mathcal G$, we denote $L^0(\mathcal G,\mathbb{R}^d)$ the space of $\mathbb{R}^d$ valued and $\mathcal G$-measurable random variables.
We denote by $L^2(\mathcal G,\mathbb{R}^d)$ the set of $X \in L^0(\mathcal G,\mathbb{R}^d)$ satisfying
\begin{equation*}
    \| X \|_{L^2(\mathcal{G},\mathbb{R}^d)} \coloneqq \mathbb E\left[|X|^2 \right]^\frac{1}{2}<+\infty.
\end{equation*}
We denote $L^\infty(\mathcal G,\mathbb{R}^d)$ the set of $X \in L^0(\mathcal G,\mathbb{R}^d)$ such that
\begin{equation*}
    \| X \|_{L^\infty(\mathcal{G})} \coloneqq \esssup_{\omega \in \Omega} \sup_{i \in \{1,\ldots,d\}}|X^i(\omega)| < + \infty.
\end{equation*}

\noindent \textit{$L^2(\mathbb G,\mathbb{R}^d)$ spaces.} 
For a filtration $\mathbb G = (\mathcal G_t)_{0\leq t\leq T}$ we denote by $L^0(\mathbb G,\mathbb{R}^d)$ the space of $\mathbb{R}^d$ valued $\mathbb G$-progressively measurable random processes. We denote $L^2(\mathbb G,\mathbb{R}^d)$ the set of $X \in L^0(\mathbb G,\mathbb{R}^d)$,
\begin{equation*}
    \| X \|_{L^2(\mathbb{G},\mathbb{R}^d)} \coloneqq \mathbb E\left[\int_0^T  |X_s|^2 \dd s\right]^{\frac{1}{2}}<+\infty.
\end{equation*}

\noindent \textit{$S^2(\mathbb G,\mathbb{R}^d)$ spaces.}
We denote $S^2(\mathbb G,\mathbb{R}^d)$ the set of $X \in L^0(\mathbb G,\mathbb{R}^d)$ satisfying
\begin{equation*}
    \|X\|_{S^2(\mathbb G,\mathbb{R}^d)}  \coloneqq  \mathbb E\left[\sup_{t \in [0,T]}|X_t|^2 \right]^{\frac{1}{2}}<+\infty.
\end{equation*}
We define $S^\infty(\mathbb G,\mathbb{R}^d)$ the set of $X \in L^0(\mathbb G,\mathbb{R}^d)$ such that 
\begin{equation*}
\|X\|_{S^\infty(\mathbb G,\mathbb{R}^d)}  \coloneqq \sup_{t \in [0,T]} \|X_t\|_{L^\infty(\mathcal{G}_t,\mathbb{R}^d)} < + \infty.
\end{equation*}

For each space defined above, we omit the notation $\mathbb{R}^d$ when $d = 1$ in the following.

\section{Technological investment model} \label{sec:tech-inv-model}

In this section we present and discuss the model under investigation in this article.  For simplicity, the main economic variables of the model (such as capital, consumption, and pollution) are expressed in normalized units and should be interpreted as dimensionless quantities. The technical assumptions and mathematical proofs are left to the next section.
In Section \ref{sec:rep-country} we introduce a representative country in the economy. In Section \ref{sec:extern} we describe the externality (or interaction) process $p$ and the notion of Nash equilibrium.

\paragraph{Stochastic context (S). \label{stochastic-context}}

We fix a time horizon $T>0$. 
We consider two complete filtered probability spaces $(\Omega^0, \cF^0,\F^0, \bP^0)$ and $(\Omega^1, \cF^1,\F^1, \bP^1)$, the first one carrying the Brownian motion $W^0$ and the second carrying the Brownian motion $W$. We then equip the product space $\Omega = \Omega^0 \times \Omega^1$ with the completion $\cF$ of the product $\sigma$-field under the product probability
measure $\bP = \bP^0\otimes\bP^1$, the extension of $\bP$ to $\cF$ being still denoted by $\bP$. Generic elements of \(\Omega\) are denoted
\(\omega=(\omega_{0},\omega_{1})\) with \(\omega_{0}\in\Omega_{0}\) and
\(\omega_{1}\in\Omega_{1}\). The right-continuous and complete augmentation of the product filtration is denoted by $\F$.  
In our setting, $W$ is an idiosyncratic noise and $W^0$ is the common noise.

\subsection{Representative country} \label{sec:rep-country}

In this section we present a representative country in the economy. To do this, we fix an external variable $p$ (which will be an externality variable), which we will discuss in the next paragraph.
We assume that there is a continuum of identical (which may be heterogeneous in terms of their parameters if they are identically and independently distributed) and atomless countries in the economy. 
The representative country optimises its utility of consumption through a vector of investments. 

\paragraph{Capital dynamics.}
We consider $n>0$ different types of capital $k^i_t$, valued in $\mathbb{R}^n$ for each time $t \in [0,T]$.
In each period, the level of each capital $i \in \{1,\ldots, n \}$ is increased by a given investment flow $a^i_t$ minus capital depreciation $\delta k^i_t$, and randomly disturbed by a Gaussian noise whose variance depends on the level of the capital itself (i.e. the noise scales with the level of the capital). More precisely, the dynamics of the capital vector $k$ is described by the following stochastic differential equation (or SDE)
\begin{eqnarray}
    \label{eq:dynamics_k}
    \dd k_t = \left(a_t-\delta  k_t \right) \dd t +\sigma(k_t)  \dd W_t, \quad k_0 = \kappa.
\end{eqnarray}
where $\delta$ is a $n \times n$ diagonal matrix of depreciation rates with positive diagonal elements, and $\sigma(k)$ is an $n \times n$ diagonal matrix of volatility rates scaled by the level of capital, i.e. 
\begin{equation}
    \sigma(k)=\left(\delta_{i,j}\sigma_i k^i\right)_{i,j = 1,...,n},
    \label{eq:sigStruc}
\end{equation}
where $\delta_{i,j}$ is the Kronecker symbol and $\kappa$ is a positive vector of the initial capital level of the representative country in the different technologies.
\begin{example}
    In a pollution model, we can consider two types of capital ($n=2$): the  brown capital $k^1$ which pollutes the environment, and the green capital $k^2$, which does not.
\end{example}

\paragraph{Investment.}
The country can invest in all types of capital. 
We assume that each country faces a type of entropic adjustment cost of investment
$$
    K(a):=-\sum_{i = 1}^n a^i \ln (a^i).
$$

The entropic cost $K(a) = -\sum_{i=1}^n a^i \ln(a^i)$ is strictly concave on $\R_+^n$, since each component function $a^i \mapsto a^i \ln(a^i)$ is strictly convex for $a^i>0$.
This form of cost prevents from negative investments (which is not the case of a quadratic penalization cost and could yield further difficulties in the analysis) and  ensures uniqueness and regularity of the optimal control.

\paragraph*{Objective of each country.}
We assume that the production $F:\R^n \times \R^d\rightarrow\R$ is a function of capital $k$ and of the total level of externality $p$ in $\R^d$. Given a stochastic trajectory of the externality $p \in L^2(\mathbb{F},\mathbb{R}_+^d)$, each country solves the optimisation problem 
\begin{equation} \label{indiv-pb} \tag{P}
    \sup_{a \in \cA(k,p)}  J[p](a) \coloneqq \mathcal{U}[p](a) + \mathcal{V}[p](a) + \theta \mathcal{K}(a),
    \end{equation}
    for some $\theta >0$. 
\begin{remark} Let $c_t$ denote the level of consumption at time $t\in [0,T]$. Using the usual macroeconomic relation \cite[equation (2.8)]{acemoglu2008introduction} for closed economy (that is to say without importation and exportation) and without government spending
\begin{equation*}
    F(k_t,p_t) =  \mathds{1}\cdot  a_t + c_t, 
\end{equation*}
and substituted into the utility function, we can see that the objective of each country is simply to maximise its actualised utility of consumption. 
\end{remark}

    The admissible controls are such that the sum of investments does not exceed production, i.e.
    \begin{displaymath}
        \cA(k,p) = \left\{a\in L^2(\F,\R^d), \; a_t \le F(k^{a}_t,p_t), \; a.s, a.e. \right\}.
    \end{displaymath}
    The first term 
    \begin{equation*} 
        \mathcal{U}[p](a) \coloneqq \mathbb{E}\left[ \int_0^T  u( F(k^a_s,p_s) - \mathds{1} \cdot a_s) e^{-\rho s} \dd s\right],
    \end{equation*}
    denotes the expected and discounted utility derived from consumption (see remark below), where $k^a$ is the controlled process solution of \eqref{eq:dynamics_k}.
    The second term 
    \begin{equation*}\mathcal{V}[p](a) \coloneqq  \mathbb{E}\left[ g(k^a_T, p_T)e^{-\rho T}\right],
    \end{equation*}
    is the expected terminal reward.
    The third term 
    \begin{equation*}
        \mathcal{K}(a) = -\mathbb{E}\left[ \int_0^T    \sum_{i = 1}^n a^i_s \ln (a^i_s) e^{-\rho s} \dd s \right],
    \end{equation*} is the adjustment cost of investment.

\subsection{Externality and equilibrium} \label{sec:extern}
In this section we describe the externality process and the notion of equilibrium. 
We start with the externality process, which is random and common to all players. 
We then introduce the notion of a mean-field game solution, i.e. the Nash equilibrium condition for this mean-field dynamic game.
\paragraph{Externality.}

The process $p$ represents the externality faced by all the countries. It is valued in $\mathbb{R}^d$, the externality of production can be of different nature (for a pollution model it can be the \ce{CO2}, \ce{CH4} concentration, mean surface temperature, level of resource stocks, biodiversity loss, etc). 
We assume that the dynamics of the externality is 
\begin{equation}
    \label{eq:dynamics_p}
    \dd p_t = \Phi(e_t,p_t)\dd t + \gamma(p_t) \dd W^0_t, \quad p_0 = \eta, \quad e_t = \E[\phi(k_t)| \cF^0_t],
\end{equation}
where $\gamma(p)$ is a $d\times d$ diagonal matrix of volatility rates scaled by the level of pollution, i.e.
\begin{displaymath}
    \gamma(p)=\left(\delta_{i,j}\gamma_i p^i\right)_{i,j = 1,...,d},
\end{displaymath}
The mapping $\phi \colon \R^n \to \mathbb{R}^d$ transforms the capital level into the externality level. 
\begin{example}
    In the context of greenhouse gas emissions, $\phi$ can be a simple linear map
    \begin{equation}
        \phi (k) = \Gamma k,
    \end{equation}
    where $\Gamma$ is a matrix with positive entries encoding $\ce{CO_2}$ equivalent emissions for each technology, for each pollutant gas.
\end{example}
For each capital, we can associate the variable $e$ which is the total contribution of each country to the total level of the externality.
The drift $\Phi:\R^n\times \R^d\rightarrow \R^d$ represents the average dynamics of the externality.
The dynamics of the externality is driven by a common noise $W^0$, creating a dependence of all countries in the game on this source of randomness. This implies that the coupling variable $p$ is a common random environment for all players.

\paragraph{Nash equilibrium.}
The mean-field game problem is to find a tuple $(\bar{a},\bar{p})$ such that 
\begin{equation} \label{Nash-eq} \tag{N}
    \bar{a} \in \argmax_{a \in \mathcal{A}(k,\bar{p})} J[\bar{p}](a), \quad \bar{p} = p^{\bar{a}}.
\end{equation}
Let us describe this condition. Given an externality process, the objective of each country is to find an optimal investment strategy to maximise its criterion. Once each country has found its optimal strategy, and thus a trajectory of capital, this prescribes a new externality process via the externality dynamics. We say that we have found a mean-field solution (or a Nash equilibrium) if the new externality process is equal to the initial one. 
In the following, we consider two kind of solutions: strong and weak solutions, see definitions \ref{def:strong} and \ref{def:weak-eq-FBDSE} below. The main difference between the two notions is that the probabilistic setting is fixed in the strong solution context, while it is part of the solution in the weak solution notion. We refer to \cite[Vol. II]{carmona2018probabilistic} for a detailed discussion and a general theory.

\section{Mathematical analysis} \label{sec:math-analysis}

In Section \ref{sec:assumptions} we state the main assumptions for our analysis. Then we start the proof of the existence and uniqueness of a Nash equilibrium.
This is done in several steps. Freezing the coupling variable, we characterize the optimal control of a representative country using a stochastic maximum principle approach in Section \ref{sec:smp}.
Section \ref{sec:regu} is dedicated to the study of the  regularity of the optimal policy. 
Section \ref{subsec:contraction} establishes the existence and uniqueness of a Nash equilibrium via a contraction argument. Section \ref{subsec:master_equation} provides a discussion on the master equation associated to the mean field game model we consider in this article. Under mild assumptions we provide the existence of weak equilibria in Section \ref{sec:weak_existence}. Section \ref{sec:uniqueness} ensures the uniqueness of equilibria under a monotone regime. 

\subsection{Assumptions} \label{sec:assumptions}

In this section we provide the main assumptions of this article.

\paragraph*{Assumptions on the utility function.} 
The utility function $u:\R_+^*\rightarrow \R$ has the following properties:
\begin{itemize}

    \item[$i)$] $u$ is increasing and strictly concave. 
    \item[$ii)$] $u$ is of class $C^2$. 
    \item[$iii)$] $u'(c)$ tends to $+\infty$ when $c$ goes to $0^+$. 
    \item[$iv)$]\label{ass:u} for every $\varepsilon>0$, $\max_{ \varepsilon\le c}\abs{u''(c)}<+\infty$. 
\end{itemize}

The first three assumptions on $u$ are classical. The last assumption requires the curvature to remain uniformly bounded away from zero, preventing explosive behavior of $u''(c)$ for positive consumption levels. Economically, this condition ensures that absolute risk aversion cannot become arbitrarily large at normal levels of consumption, allowing extreme sensitivity to consumption only in the vicinity of zero, where subsistence considerations naturally justify such behaviour.

\paragraph*{Assumptions on the production function.} The production function  $F:\R_+^n\times \R_+^d \rightarrow \R$ satisfies:
\begin{itemize}
    \item[$i)$] $F$ is bounded by below by a positive constant $\underline{F}$.
    \item[$ii)$] $F$ is of class $C^{1,1}$.
    \item[$iii)$] $F$ is increasing and concave with respect to its first variable.
    \item[$iv)$] $F$ is globally Lipschitz.  
\end{itemize}

\paragraph*{Assumptions on the terminal cost.}
We assume that the terminal cost $g: \R_+^n \times \R_+^d \rightarrow \R$ satisfies:
\begin{itemize}
    \item[$i)$] $g$ is non negative. 
    \item[$ii)$] $g$ is increasing and concave with respect to its first variable.
    \item[$iii)$] $g$ is of class $C^{1,1}$.
\end{itemize}

\paragraph*{Assumptions on the function transforming the level of capital into externalities.}

We assume that $\phi:\R_+^n\rightarrow\R_+^d$ is Lipschitz, i.e. there exists $C_\phi$ such that 
\begin{displaymath}
    \abs{\phi(k^1) - \phi(k^2)}\le C_\phi \abs{k^1 - k^2},\quad \forall k^1,k^2\in \R_+^n.
\end{displaymath}

\paragraph*{Assumptions on the drift of the externality.}
We assume that $\Phi:\R_+^n\rightarrow\R_+^d$ is Lipschitz, i.e. there exist $C_{\Phi,e}$ and $C_{\Phi,p}$ such that 
\begin{displaymath}
  \abs{\Phi(e^1,p^1) - \Phi(e^2,p^2)}\le C_{\Phi,e}\abs{e^1 -e^2} + C_{\Phi,p}\abs{p^1 -p^2},\quad \forall k^1,k^2\in \R_+^n.
\end{displaymath}

\paragraph{Examples of utility and production functions.}
\begin{enumerate}
    \item The classical Constant Relative Risk Aversion (CRRA) utility, which is often used in economic applications, satisfies these assumptions. More precisely, for any $c>0$ the function
\begin{displaymath}
    u(c) = \left\{\begin{array}{ll}
       \frac{c^{1 - \eta}}{1 -\eta},  &\text{if }\eta \in(0,1), \\
        \ln(c) & \text{if }\eta = 1,
    \end{array}\right.
\end{displaymath}
where the parameter $\eta$ measures the risk aversion, satisfies the hypothesis. 

    \item The assumptions on the production function are more restrictive: it follows from the required Lipschitz properties. 
    
    Here is an example of a production function for $n$-sectors that fits our assumptions:
    \begin{displaymath}
    F(k,p) = A\left(\sum_{i= 1}^N K_i(k_i,p)^\gamma\right)^{\beta/\gamma} - \sum_{i=1}^N b_i(p)K_i(k_i,p)
\end{displaymath}
where 
\begin{displaymath}
    K_i(k,p)=\min(k+\varepsilon, k^*_i(p))
\end{displaymath}
and $k^*_i$ satisfies \begin{displaymath}
    \nabla_kF(k^*,p) = 0 \Leftrightarrow k_i^* = \left(\frac{A\beta}{b_i(p)}\right)^\frac{1}{1-\gamma}\left(\sum_{j=1}^Nk_j^*\right)^\frac{\beta-\gamma}{\gamma(1-\gamma)},\quad \forall i=1,...,N.
\end{displaymath}
The constant $A$ denotes the total productivity factor, and $b_i(p)$ represents the cost of employing one unit of capital in that sector. They are taken positive and regular. We assume that, even in the absence of capital, there remains a (small) positive level of output $\varepsilon$. Moreover, we do not utilise capital when its marginal cost exceeds its marginal product, hence the introduction of $K_i$.

\end{enumerate}

\subsection{Stochastic maximum principle} \label{sec:smp}

Given a trajectory of the externality $p$, we characterize the optimal control using the stochastic maximum principle (see, e.g., \cite[Section 6.4.2]{pham2009continuous}). To this end, we introduce the generalized Hamiltonian associated with the optimal control problem:
\begin{displaymath}
    H(a,k,p,y,z) = [a-(\delta+\rho) k ]y + u(F(k,p) - a\cdot \ind) - \theta\sum_{i=1}^n a^i\ln(a^i) + \Tr(\sigma(k)^\top  z).
\end{displaymath}

The concavity of $H$ with respect to $(a,k)$ plays a central role for the condition to be sufficient. Indeed, the Hamiltonian is jointly concave with respect to $(a,k)$: the term $a\cdot y$ is linear in $(a,k)$, the function $u(F(k,p)-a\cdot\mathds{1})$ is concave in $(a,k)$ since $u$ is assumed concave and non-decreasing and $F(\cdot,p)$ is concave in $k$, while $a\mapsto -\theta\sum_{i=1}^n a^i\ln(a^i)$ is strictly concave on $\mathbb{R}_+^n$. The remaining terms are linear in $k$. Therefore the sum of these terms is jointly concave in $(a,k)$.

For ease of notation, let us introduce $A(k,p)$ the set of the possible values of a player's investment when its state is $k$ and the externality is $p$, i.e. 
\begin{displaymath}
    A(k,p) = \left\{a \in \R_+^n\;:\; \ind\cdot a \le F(k,p)\right\}.
\end{displaymath}
We now state the stochastic maximum principle.

\begin{theorem}[Stochastic Maximum Principle]
\label{thm:SMP}
Let $p$ be given. 
If $(k^*,y^*,z^*,z^{0,*},a^*)$ is a solution to
\begin{equation}
\label{eq:bsde_y_SMP}
\left\{
\begin{array}{rll}
-\dd y^*_t &= \nabla_k H(a^*_t,k^*_t,p_t,y^*_t,z^*_t)\,\dd t 
- z^*_t \dd W_t - z^{0,*}_t \dd W^0_t, 
& y^*_T = \nabla_k g(k^*_T,p_T),\\[0.5em]
\dd k^*_t &= (a^*_t - \delta k^*_t)\,\dd t + \sigma(k^*_t)\,\dd W_t, 
& k^*_0 = \kappa,
\end{array}
\right.
\end{equation}
with 
\begin{equation}
\label{eq:a-opt}
H(a^*_t,k^*_t,p_t,y^*_t,z^*_t)
=
\max_{a \in \mathcal{A}(k^*_t,p_t)} 
H(a,k^*_t,p_t,y^*_t,z^*_t),
\quad \dd t \otimes \dd \bP\text{-a.s.}
\end{equation}
then $a^* \in \mathcal{A}$ is a solution to \eqref{indiv-pb} and Problem \eqref{indiv-pb} admits at most one solution. 
\end{theorem}

\begin{proof}

The scheme of the proof of this step is standard and can be found in \cite[Section 6.4.2]{pham2009continuous}. The key of the proof relies on the strict concavity of the map $(a,k) \mapsto H(a,k,p,y,z)$. 
    Let us consider another admissible control $a$ and its controlled process $k$.
    By definition of the criteria we have
    \begin{equation*}
        J[p](a) -  J[p](a^*) = \mathcal{U}[p](a)  - \mathcal{U}[p](a^*) + \mathcal{V}[p](a) -  \mathcal{V}[p](a^*) - \theta( \mathcal{K}(a) - \mathcal{K}(a^*)).
    \end{equation*}
    On the one hand, let us compute 
    \begin{align*}
        I_1 &= 
        \mathcal{U}[p](a)  - \mathcal{U}[p](a^*)  - \theta (\mathcal{K}(a) - \mathcal{K}(a^*))  \\
        &=\mathbb{E}\left[\int_0^T  (u(F(k_t,p_t)-a_t\cdot\ind) -  u(F(k_t^*,p_t)-a^*_t\cdot\ind) )e^{-\rho t} \dd t \right]\\
        &\quad +\theta \mathbb{E}\left[\int_0^T  (K(a_t)-K(a_t^*))e^{-\rho t}\dd t\right]\\
        &= \mathbb{E}\left[\int_0^T  (H(a_t,k_t,p_t,y_t^*,z_t^*) - H(a_t^*,k_t^*,p_t,y_t^*,z_t^*)  e^{- \rho t} \dd t\right]\\
        &\quad- \mathbb{E}\left[\int_0^T   [(a_t - a_t^* - (\delta +\rho) (k_t-k_t^*))y_t^* + \Tr((\sigma(k_t) -\sigma(k_t^*))^\top  z_t^*) ] e^{-\rho t} \dd t \right].
    \end{align*}
    On the other hand, by concavity of the terminal condition with respect to its first variable, by the It{\^o} formula we have 
    \begin{align*}
        I_2 & =  \mathcal{V}[p](a) -  \mathcal{V}[p](a^*) \le  e^{-\rho T}\mathbb{E}\left[\nabla_k  g(k^*_T, p_T) \cdot (k_T - k_T^*)  \right] =   \mathbb{E}\left[  e^{-\rho T} y^*_T \cdot (k_T - k_T^*)  \right] \\
        & =  \mathbb{E}\left[\int_0^T    e^{- \rho t} [ - \nabla_k  H(a^*_t,k^*_t,p_t,y^*_t,z^*_t) -\rho y^*_t] \cdot(k_t - k_t^*) \dd t \right]\\
        &\quad+ \mathbb{E}\left[\int_0^T   e^{- \rho t} [ y^*_t(a_t - a_t^* - \delta (k_t-k_t^*)) +\Tr((\sigma(k_t) -\sigma(k_t^*))^\top  z_t^*) ]\dd t\right].
    \end{align*}
    Finally, we have 
    \begin{align*}
        J[p](a) -  J[p](a^*) & =  I_1 + I_2\\
        &\leq  \mathbb{E}\left[\int_0^T   e^{- \rho t} [ H(a_t,k_t,p_t,y_t^*,z_t^*) - H(a_t^*,k_t^*,p_t,y_t^*,z_t^*) ]\dd t\right] \\
        &\quad - \mathbb{E}\left[\int_0^T    e^{- \rho t} \nabla_k H(a^*_t,k^*_t,p_t,y_t^*,z_t^*) \cdot(k_t - k_t^*) \dd t \right] \\
        & < 0,
    \end{align*}
    where we used \eqref{eq:a-opt} and the strict concavity of the Hamiltonian with respect to its two first variables. 
    
\end{proof}

\begin{remark}
An equivalent formulation consists in enlarging each country’s private state to $(k,p)$ and replacing the conditional expectation in $\Phi$ by a given $\mathbb{F}^0$–adapted process $e_t$, representing the random environment. In that case, one works with the modified Hamiltonian
\begin{displaymath}
\tilde H\bigl(a,k,p,y,z,r,z^{r},e\bigr)
=
H\bigl(a,k,p,y,z\bigr)
+
r\,\Phi\bigl(e,p\bigr)
+
\trace\bigl[\gamma(p)^\top z^{r}\bigr].
\end{displaymath}

The adjoint process $r_t$ is then the co-state variable associated with $p_t$, that is, it represents the marginal value (or social cost of carbon) of a variation in the externality $p$. The full adjoint system becomes
\begin{displaymath}
\begin{cases}
(y,z,z^0)\text{ solves the usual BSDE \eqref{eq:bsde_y_SMP},}\\[6pt]
- \dd r_t
=
\nabla_p \tilde H\bigl(a^*_t,k^*_t,p_t,y_t,z_t,r_t,z^r_t,e_t\bigr)\,\dd t
-
z^r_t\,\dd W_t
-
z^{0,r}_t\,\dd W^0_t,
\\[6pt]
r_T = \nabla_p g\bigl(k_T,p_T\bigr).
\end{cases}
\end{displaymath}

However, in our framework the process $p$ is exogenous from the perspective of each representative country: its dynamics does not depend on the individual state $k$ nor on the control $a$. As a consequence, although one may formally introduce the adjoint variable $r_t$, it does not enter the first-order condition characterizing the optimal control $a^*_t$. In particular, the maximization of the Hamiltonian with respect to $a$ depends only on $(k,y,z)$ and not on $r$.

Therefore, freezing $p$ and treating it as a given external process allows us to reduce the problem and avoid manipulating a more involved forward–backward system including the additional adjoint variable $r$, without any loss of information for the characterization of the optimal control. This precisely reflects the negligible player assumption: individual countries do not internalize their infinitesimal impact on $p$. Moreover, since $p$ is not controlled at the individual level, no additional concavity condition in $p$ is required.
\end{remark}

\subsection{Regularity of the optimal control variable} \label{sec:regu}
By the stochastic maximum principle, the optimal control is a mapping $a \colon \mathbb{R}^n \times  \mathbb{R}^d \times \mathbb{R}^n \to \mathbb{R}^n$,
\begin{equation} \label{def:opt-policy}
    a(k,p,y) = \argmax_{\alpha \in A(k,p)} H(\alpha,k,p,y,z).
\end{equation}
As we shall see, the right-hand side is independent of $z$, so we omit the dependence on this variable in the left-hand side.
This section is devoted to the proof of the following proposition.
\begin{proposition} \label{prop:a-Hk-Lip-maps}
    The mappings 
    \begin{align} \label{maps:a}
        &\R^n\times \R^d\times [0,\overline{y}]^n \ni(k,p,y) \mapsto a(k,p,y), \\
        & \R^n\times \R^d\times [0,\overline{y}]^n\times \R^{n\times n} \ni(k,p,y,z) \mapsto \nabla_k H(a(k,p,y),k,p,y,z), \label{maps:H}
    \end{align} 
    are Lipschitz continuous. 
\end{proposition}
The proof can be found at the end of this section. Let us first justify that the policy \eqref{def:opt-policy} is independent on $z$.
By a direct computation, it can be seen that every $\bar{\alpha} \in \mathbb{R}^n$ which satisfies the first order condition $\nabla_a H(\bar{\alpha},k,p,y,z) = 0$ is such that
\begin{equation*} 
    \label{eq:control}
    \bar{\alpha}^i = \exp\left(\frac{1}{\theta}\left(y^i - u'(F(k,p) - \mathds{1} \cdot \bar{\alpha} )-1\right)\right),
\end{equation*}
for all $i \in \{1,\ldots,n\}$ which is independent of $z$.
For all $(k,p,y) \in \mathbb{R}^n \times \mathbb{R}^d \times \mathbb{R}^n$, we define the mapping $f[k,p,y] \colon \mathbb{R}^n \to \mathbb{R}$,
\begin{equation*}
    f[k,p,y](\xi) = \xi - \sum_{i=1}^n\exp\left(\frac{1}{\theta}\left(y^i - u'(F(k,p) - \xi)\right)-1\right),
\end{equation*}
and the equation
\begin{equation}
    \label{eq:norm_1_control}
    f[k,p,y](\xi) = 0. 
\end{equation}
Finally, provided that there exists a unique $\xi(k,p,y)$ solution to the last equation (which is the purpose of the next lemma), the optimal policy defined by \eqref{def:opt-policy} is given by
\begin{equation} \label{eq:opt-policy-ai}
    a^i (k,p,y) = \exp\left(\frac{1}{\theta}\left(y^i - u'(F(k,p) -  \xi(k,p,y))\right)-1\right),
\end{equation}
for each $i \in \{1,\ldots,n\}$.
\begin{lemma} \label{lemma:xi-F}
    There exists a unique solution $\xi(k,p,y)$ to equation \eqref{eq:norm_1_control}, satisfying 
    \begin{displaymath}
        0 <  \xi(k,p,y) < F(k,p),
    \end{displaymath}
    for each $(k,p,y) \in \mathbb{R}^n \times \mathbb{R}^d \times \mathbb{R}^n$.
\end{lemma}
\begin{proof}
    We can verify that for any $(k,p,y) \in \mathbb{R}^n \times \mathbb{R}^d \times \mathbb{R}^n$,
    \begin{equation}
        f[k,p,y](0) <0, \quad \lim_{\xi \to F(k,p)} f[k,p,y](\xi) = F(k,p) \geq \underline{F} >0,
    \end{equation}
    since we assume that $F$ is lower bounded. The intermediate value theorem gives the existence of a solution to the equation \eqref{eq:norm_1_control} valued in $(0,F(k,p))$. The uniqueness comes from the strict monotonicity of $f[k,p,y]$.
    
\end{proof}

From now on we define $\xi$ as the solution of the equation \eqref{eq:norm_1_control} and the consumption strategy is given by $c(k,p,y) = F(k,p) - \xi(k,p,y)$ for all $(k,p,y) \in \R_+^n \times \R_+^d \times \R_+^n$.

\begin{lemma} \label{lemma:c-bounded-below}
    For any positive real number $\overline{y}$. The mapping
    \begin{equation}
        \R_+^n\times \R_+^d\times [0,\overline{y}]^n  \ni (k,p,y) \mapsto c(k,p,y),
    \end{equation}
    is bounded below by a positive constant $\eta_{\overline{y}}$. 
\end{lemma}
\begin{proof}
    Fix $(k,p,y)\in \R_+^n\times \R_+^d\times [0,\overline{y}]^n$, by definition of $\xi$ we have
    \begin{equation*}
        \xi(k,p,y) \le \exp\left(\frac{1}{\theta}\left(\overline{y} - u'(F(k,p) -  \xi(k,p,y))\right)-1\right).
    \end{equation*}
    Then using $c = F-\xi$, we have
    \begin{equation*}
        c(k,p,y) \geq F(k,p) - \exp\left(\frac{1}{\theta}\left(\overline{y} - u'(c(k,p,y))\right)-1\right).
    \end{equation*}
    We then obtain by assumption on the lower bound of $F$ that
    \begin{equation*}
        c(k,p,y) +   \exp\left(\frac{1}{\theta}\left(\overline{y} - u'(c(k,p,y))\right)-1\right) \geq \underline{F}.
    \end{equation*}
    Since $c +  \exp\left(\frac{1}{\theta}\left(\overline{y} - u'(c)\right)-1\right) $ goes to $0$ when $c$ tends to $0$ and $\underline{F}>0$, then $c$ admits a positive bound from below, which concludes the proof.
\end{proof}

\begin{lemma} \label{lemma:xi-C1}
    The function $\xi$ is of class $C^1$.
\end{lemma}
\begin{proof}
First observe that the mapping $(\xi,k,p,y) \mapsto f[k,p,y](\xi)$ is of class $C^1$. Then
\begin{displaymath}
    \partial_{\xi} f[k,p,y](\xi) = 1-\sum_{i=1}^n\frac{1}{\theta}\exp\left(\frac{1}{\theta}\left(y^i - u'(F(k,p) - \xi)\right)-1\right)u''(F(k,p)-\xi)\ge 1. 
\end{displaymath}
Therefore, we can apply the implicit mapping theorem to deduce that $\xi$ is of class $C^1$. Moreover, 
\begin{displaymath}
    \nabla \xi(k,p,y) =  -\partial_{\xi} f[k,p,y](\xi(k,p,y))^{-1} \nabla_{(k,p,y)}f[k,p,y](\xi(k,p,y)), 
\end{displaymath}
where
\begin{displaymath}
    \nabla_{(k,p)} f[k,p,y](\xi) = \sum_{i=1}^n\frac{1}{\theta}\exp\left(\frac{1}{\theta}\left(y^i - u'(F(k,p) - \xi)\right) -1\right) u''(F(k,p) - \xi)\nabla_{(k,p)}F(k,p),
\end{displaymath}
and 
\begin{displaymath}
    \nabla_y f[k,p,y](\xi) =  - \left(\frac{1}{\theta}\exp\left(\frac{1}{\theta}\left(y^i - u'(F(k,p) - \xi)\right)-1\right) \right)_{i=1,...,n}.
\end{displaymath}
\end{proof}
\begin{corollary}
    \label{cor:bounds_gradient_xi}
    For any $(k,p,y)\in \R_+^n\times \R^d_+\times \R^n$,
    \begin{displaymath}
        \abs{\nabla_{k}\xi(k,p,y)}\le \abs{\nabla_k F(k,p)},\quad\abs{\nabla_{p}\xi(k,p,y)}\le \abs{\nabla_p F(k,p)}
    \end{displaymath}
    and
    \begin{displaymath}
        \abs{\nabla_{y}\xi(k,p,y) u''(F(k,p)-\xi(k,p,y))}\le 1.
    \end{displaymath}
\end{corollary}
Finally, we prove the Proposition \ref{prop:a-Hk-Lip-maps}. We start with the mapping \eqref{maps:a}. We note that $c = F - \xi$. We check that $u'(c)$ is Lipschitz. To do this, we show that its gradient 
\begin{equation*}
    \nabla_{(k,p,y)} u'(c) = u''(c) \nabla_{(k,p,y)} c,
\end{equation*}
is bounded. 
Indeed, Lemma \ref{lemma:c-bounded-below} gives a bound by below of $c$. Then, we deduce that $u''(c)$ is bounded. This bound and Corollary \ref{cor:bounds_gradient_xi} imply that $u''(c)\nabla_{(k,p,y)}c$ is bounded, therefore $u'(c)$ is Lipschitz.
\newline
By equation \eqref{eq:opt-policy-ai}, the optimal control writes $ a^i (k,p,y) = \exp\left(\frac{1}{\theta}(y^i - u'(c(k,p,y)))-1\right)$, for each $i \in \{1,\ldots,n\}$, so the mapping \eqref{maps:a} is Lipschitz. 
\medskip

We now turn to \eqref{maps:H}. 
Using equation \eqref{eq:sigStruc}, defining $\hat \sigma$ such that $\hat \sigma_{i,j}= \sigma_i \delta_{i,j}$, $i=1, \ldots n$, $j=1, \ldots n$,
\begin{equation*}
    \nabla_k H(a(k,p,y),k,p,y,z) = - (\delta+\rho) y + \diag(\hat \sigma z ) + u'(c(k,p,y)) \nabla_k F(k,p).
\end{equation*}
The first two terms are linear, so we only need to show that the last term is Lipschitz. We have already established that $u'(c)$ is Lipschitz. Let us check that it is bounded. From Lemma \ref{lemma:c-bounded-below} and the fact that $u'$ is decreasing and positive, we get 
\begin{equation*}
    0 < u'(c) \leq u'(\eta_{\overline{y}})<+\infty. 
\end{equation*} On the other hand, by assumptions $\nabla_kF(k,p)$ is Lipschitz and bounded. Therefore, the product function
\begin{displaymath}
    (k,p,y)\mapsto u'(c(k,p,y))\nabla_kF(k,p),
\end{displaymath}
is Lipschitz, which ends the proof.

\subsection{Strong existence and uniqueness of a mean-field game equilibrium} 

\label{subsec:contraction}

We are now in a position to tackle the mean field problem \eqref{Nash-eq}. By Theorem \ref{thm:SMP}, a tuple $(a,p)$ is solution to \eqref{Nash-eq}, if and only if there exists a solution $(k,p,y,z,z^0)$ to the following system
\begin{equation}
    \label{main:FBSDE-NE}
    \left\{
    \begin{array}{rll}
        \dd p_t & = \Phi(\E[\phi(k_t)| \cF^0_t],p_t)\dd t + \gamma(p_t) \dd W^0_t, & p_0 = \eta,\\[0.5em]
        - \dd y_t  & =  \nabla_k H(a(k_t,p_t,y_t),k_t,p_t,y_t,z_t)  \dd t - z_t \dd W_t - z^0_t \dd W^0_t, & y_T = \nabla_k g(k_T,p_T), \\[0.5em]
        \dd k_t & = (a(k_t,p_t,y_t) - \delta k_t) \dd t +  \sigma(k_t) \dd W_t,  & k_0 = \kappa,  
    \end{array}
    \right.
\end{equation}
where $a$ is the mapping defined in \eqref{def:opt-policy} and examined in the previous section (which we recall here $a(k,p,y) = \argmax_{\alpha \in \mathbb{R}^n} H(\alpha,k,p,y,z)$). 

\begin{definition} \label{def:strong}
    We say that the MFG problem \eqref{Nash-eq} admits a strong equilibrium if there exists $(k,p,y,z,z^0)$ solution to \eqref{main:FBSDE-NE}.
\end{definition}

Let us reformulate the system above as a fixed-point problem through a sequence of mappings 
$\Theta_1,\Theta_2,\Theta_3$ and $\Theta$:
\begin{enumerate}
    \item $\Theta_1\colon  L^2(\F,\R^n)\rightarrow  L^2(\F,\R^d)$,
    \item $\Theta_2\colon  L^2(\F,\R^n)\times  L^2(\F,\R^d) \rightarrow  L^2(\F,\R^{n}) \times L^2(\F,\mathbb{R}^{n \times n}) \times L^2(\F,\R^{n \times d})$,
    \item $\Theta_3\colon  L^2(\F,\R^d)\times L^2(\F,\R^n)\rightarrow L^2(\F,\R^n)$,
    \item $\Theta \colon  L^2(\F,\R^n) \to L^2(\F,\R^n)$.
\end{enumerate} 

The idea is to solve the coupled system sequentially. Given a candidate trajectory of the state variable $k$, the mappings are interpreted as follows:
\begin{enumerate}
    \item First, we determine the externality process by setting $p = \Theta_1(k)$, which corresponds to solving the first equation of the system.
    
    \item Given $(k,p)$, we solve the backward equation and obtain the adjoint variables $(y,z,z^0) = \Theta_2(k,p)$.
    
    \item Using the variables $(p,y)$, we then solve the forward equation defining the state dynamics and set $k' = \Theta_3(p,y)$.
    
    \item Finally, we define the global mapping $\Theta(k) = k'$. A solution to the system corresponds to a fixed point of this mapping, i.e. $\Theta(k) = k$.
\end{enumerate}
Therefore, finding a solution $(y,z,k,p)$ to the system of equations \eqref{main:FBSDE-NE} is equivalent to finding a fixed point of the mapping $\Theta$. The following propositions establish the well-posedness and the Lipschitz properties of the maps $\Theta_i$ ($i=1,2,3$). Under some appropriate assumptions the map $\Theta$ is a contraction. Therefore, the existence and uniqueness follow from Picard's fixed point theorem. 

\begin{proposition}
    \label{prop:pollution}
    The mapping $\Theta_1$
    is well-defined and 
    \begin{displaymath}
        \norm{\Theta_1(k) - \Theta_1(k')}_{L^2(\F,\R^d)} \le \sqrt{C_1}\norm{k-k'}_{L^2(\F,\R^n)},
    \end{displaymath}
    where $ C_1 \coloneqq C_{\Phi,e} C_\phi T e^{\left(C_{\Phi,e} + 2 C_{\Phi,p}+C_\gamma^2\right)T}$, with  $C_\gamma = \max_i \gamma_i$.
\end{proposition}
\begin{proof}
    \textit{Step 1: Well-posedness of $\Theta_1$.}
    For any $k\in L^2(\F,\R^n)$, the existence and uniqueness of the solution of the first equation of \eqref{main:FBSDE-NE} follows from the Lipschitz properties of the coefficients. Since $\norm{\eta}_{L^2(\Omega)}<+\infty$, the solution belongs to $L^2(\F,\R^d)$. Therefore $\Theta_1$ is well defined and completes the step. 
    \medskip
    
    \noindent \textit{Step 2: Lipschitz continuity.}
    Let us check that the map $\Theta_1$ is Lipschitz.  It follows from the Lipschitz properties of $\Phi$ and $\Gamma$. Consider $k^1$ and $k^2$, two elements of $L^2(\F,\R^n)$. For convenience we denote $p^i = \Theta_1(k^i)$, $\Delta p = p^1-p^2$ and $\Delta k = k^1-k^2$. We observe that $\Delta p_0 = 0$ and
    \begin{displaymath}
        \dd \Delta p_t = \left(\Phi(\E[\phi(k_t^1)\vert \cF^0_t],p^1_t) - \Phi(\E[\phi(k_t^2)\vert \cF^0_t],p^2_t)\right) \dd t + \gamma(\Delta p_t) \dd W^0_t.
    \end{displaymath}
    Therefore, for any $t \in [0,T]$, Ito's formula yields that 
    \begin{align*}
        \norm{\Delta p_t}^2_{L^2(\mathcal{F}_t,\mathbb{R}^d)} =& \E\left[\int_0^T 2\Delta p_s \left(\Phi(\E[\phi(k_s^1)\vert \cF^0_s],p^1_s) - \Phi(\E[\phi(k_s^2)\vert \cF^0_s],p^2_s)\right)\dd s \right] \\ & + \E\left[\int_0^T  \Tr(\gamma(\Delta p_s)^\top  \gamma(\Delta p_s))\dd s\right].
    \end{align*}
    Using Fubini's theorem and the Cauchy-Schwarz inequality, we have 
    \begin{align*}
        \norm{\Delta p_t}^2_{L^2(\mathcal{F}_t,\mathbb{R}^d)} 
        &\leq 2 \int_0^T  \norm{\Delta p_s}_{L^2(\mathcal{F}_s,\mathbb{R}^d)} \norm{\Phi(\E[\phi(k_s^1)\vert \cF^0_s],p^1_s) - \Phi(\E[\phi(k_s^2)\vert \cF^0_s],p^2_s)}_{L^2(\mathcal{F}_s,\mathbb{R}^d)} \dd s \\
        & +  C_\gamma^2 \int_0^T \norm{\Delta p_s}_{L^2(\mathcal{F}_s,\mathbb{R}^d)}^2 \dd s.
    \end{align*}
     Then, 
    using the Lipschitz continuity of $\Phi$ we have
    \begin{align*} \Vert\Phi(\E[\phi(k_s^1)\vert \cF^0_s],p^1_s)& - \Phi(\E[\phi(k_s^2)\vert \cF^0_s],p^2_s)\Vert_{L^2(\mathcal{F}_s,\mathbb{R}^d)}  \\[0.5em]
     &\leq C_{\Phi,p}\norm{\Delta p_s}_{L^2(\mathcal{F}_s,\mathbb{R}^d)}+ C_{\Phi,e}\norm{\E[\phi(k_s^2)\vert \cF^0_s] - \E[\phi(k_s^2)\vert \cF^0_s]}_{L^2(\mathcal{F}_s,\mathbb{R}^d)},
    \end{align*} 
    yielding 
    \begin{align*}
        \norm{\Delta p_t}^2_{L^2(\mathcal{F}_t,\mathbb{R}^d)} 
        \leq & 2 C_{\Phi,e} \int_0^T   \norm{\Delta p_s}_{L^2(\mathcal{F}_s,\mathbb{R}^d)}\norm{\phi(k_s^1)- \phi(k_s^2)}_{L^2(\mathcal{F}_s,\mathbb{R}^d)} \dd s \\
        & + 2 C_{\Phi,p} \int_0^T  \norm{\Delta p_s}^2_{L^2(\mathcal{F}_s,\mathbb{R}^d)} \dd s +C_\gamma^2 \int_0^T  \norm{\Delta p_s}_{L^2(\mathcal{F}_s,\mathbb{R}^d)}^2 \dd s.
    \end{align*}
    Since $\phi$ is assumed to be Lipschitz and conditional expectation is a contraction operator in $L^2(\F,\R^n)$, the mapping $ L^2(\F,\R^n) \ni k \mapsto \mathbb{E}[\phi(k)\mid\mathcal{F}^0_t]$ is also Lipschitz, we finally end up with
    \begin{align*}
        \norm{\Delta p_t}^2_{L^2(\mathcal{F}_t,\mathbb{R}^d)} \leq \left(C_{\Phi,e} + 2 C_{\Phi,p}+C_\gamma^2\right)\int_0^T \norm{\Delta p_s}_{L^2(\mathcal{F}_s,\mathbb{R}^d)}^2 \dd s + C_{\Phi,e}C_\phi\int_0^T \norm{\Delta k_s}_{L^2(\mathcal{F}_s,\mathbb{R}^d)}^2 \dd s.
    \end{align*}
    Thus, Gr{\"o}nwall's Lemma leads to 
    \begin{displaymath}
        \norm{\Delta p_t}^2_{L^2(\mathcal{F}_t,\mathbb{R}^d)} \le  C_{\Phi,e} C_\phi e^{\left(C_{\Phi,e} + 2 C_{\Phi,p}+C_\gamma^2\right)t} \int_0^T \norm{\Delta k_s }_{L^2(\mathcal{F}_s,\mathbb{R}^d)}^2 \dd s,
    \end{displaymath}
    for any $t \in [0,T]$. 
    Taking the integral over time, Fubini's theorem finally gives us
    \begin{displaymath}
        \norm{\Delta p}^2_{ L^2(\F, \mathbb{R}^d)} \le C_1 \norm{\Delta k}^2_{L^2(\F,\mathbb{R}^n)},
    \end{displaymath}
    concluding the step and the proof.
\end{proof}

We now turn to $\Theta_2$. For ease of notation, we introduce the function $\upsilon \colon \R^n_+\times \R^d_+\times \R^n \to \mathbb{R}$,
    \begin{displaymath}
        \upsilon(k,p,y) = u'(c(k,p,y))\nabla_k F(k,p).
    \end{displaymath}
We recall that when $y$ is bounded, $\upsilon$ is Lipschitz (see the end of the proof of Proposition \ref{prop:a-Hk-Lip-maps}) and we denote $C_{\upsilon,k}$, $C_{\upsilon,p}$ and $C_{\upsilon,y}$ its Lipschitz constants with respect to $k$, $p$ and $y$.
   
\bigskip 

\begin{proposition}
    \label{prop:adjoint}
    The mapping $ \Theta_2$
    is well-defined. There exists $C_y > 0$ such that for any $(k,p)$
    \begin{displaymath}
        \norm{\Theta^1_2(k,p)}_{S^\infty(\F, \mathbb{R}^n)}\le C_y,
    \end{displaymath}
    where $\Theta^1_2(k,p)$ denotes the first component of $\Theta_2(k,p)$.
    In addition,
    \begin{displaymath}
        \norm{\Theta^1_2(k^1,p^1) - \Theta^1_2(k^2,p^2)}_{L^2(\F, \mathbb{R}^n)}^2 \le C_2\norm{k^1-k^2}_{L^2(\F, \mathbb{R}^n)}^2+C_3\norm{p^1-p^2}_{L^2(\F, \mathbb{R}^d)}^2,
    \end{displaymath}
    where 
    \begin{align*}
        C_2 &=(C_{\nabla_k g,k}^2 + T C_{\upsilon,k}^2) \frac{1}{\nu} \left(e^{\nu T} - 1\right) , \\
        C_3 &= (C_{\nabla_k g,p}^2 + T C_{\upsilon,p}^2)\frac{1}{\nu} \left(e^{\nu T} - 1\right) ,
    \end{align*}
    with $\nu = - 2(\delta + \rho) + C_{\upsilon,y}^2 + C_\sigma^2 + C_{\upsilon,k} +  C_{\upsilon,p}$, with  $C_\sigma = \max_i \sigma_i$.
\end{proposition}

\begin{proof}
    \textit{Step 1: Well-posedness of $\Theta_2$.}
    The Lipschitz properties of the coefficients ensure well-posedness of the BSDE (see \cite[Chapter 4]{zhang2017backward}).
    
    \medskip
    
    \noindent 
    \textit{Step 2: $S^\infty$ estimate. } 
    Using standard arguments, we start by linearising the BSDE. The fundamental theorem of calculus gives us 
    \begin{align*}
        u'(F(k_t,p_t) - \xi^*(k_t,p_t,y_t)) &= u'(F(k_t,p_t) - \xi^*(k_t,p_t,0)) \\
        &\quad\quad- \int_0^1u''(F(k_t,p_t) - \xi^*(k_t,p_t,sy_t))\nabla_y\xi^*(k_t,p_t,sy_t) \dd s y_t.
    \end{align*}
    Therefore, by setting
    \begin{displaymath}
        \alpha_t = -(\delta+\rho)- \int_0^1u''(F(k_t,p_t) - \xi^*(k_t,p_t,sy_t))\nabla_y\xi^*(k_t,p_t,sy_t) \dd s\nabla_k F(k_t,p_t),
    \end{displaymath}
    we observe that $y$ satisfies
    \begin{displaymath}
        -\dd y_t = (u'(F(k_t,p_t) - \xi^*(k_t,p_t,0))\nabla_k F(k_t,p_t)+\alpha_t  y_t + \diag(\hat \sigma z_t))  \dd t - z_t \dd W_t - z^0_t \dd W_t^0.
    \end{displaymath}
    Note that using Corollary \ref{cor:bounds_gradient_xi}, $\norm{\alpha}_{S^\infty(\mathbb{F},\mathbb{R}^n)}\le - \delta - \rho + \norm{\nabla_k F}_{\infty}$.
    \\
    Let $\Q$ be the equivalent probability measure to $\bP$ given by $\dd \Q = \mathcal{E}(\int \diag(\hat \sigma) \dd W)_T \dd \bP$ where $\mathcal{E}(\int \diag(\hat \sigma) \dd W)_T$ denotes the stochastic exponential associated with  $\diag(\hat \sigma)$. Note that Novikov's condition is trivially satisfied since $\hat\sigma$ is assumed to be constant. 
  
    We deduce from Ito's formula and by taking the conditional expectation under $\Q$, that for any $t \in [0,T]$ we have

    \begin{align*}
        y_te^{-\int_t^T \alpha_udu}=\E_\Q\left[ \nabla_kg(k_T,p_T) + \int_t^T u'(F(k_\tau,p_\tau) - \xi^*(k_\tau,p_\tau,0))\nabla_k F(k_t,p_t)e^{-\int_\tau^T \alpha_u \dd u}\dd \tau\bigg\vert \cF_t\right].
    \end{align*}
    
    We observe that all components of $y_t$ are non-negative and bounded. Indeed, $\alpha$ is bounded in $S^\infty(\F)$, $\nabla_kg$ and $\nabla_k F$ are uniformly bounded and non-negative, so that 
    \begin{displaymath}
        (k,p)\mapsto u'(F(k,p) - \xi^*(k,p,0)),
    \end{displaymath}
    from the bound given in Lemma \ref{lemma:c-bounded-below} (with $\overline y = 0$) and the concavity of $u$. Finally, for any $t \in [0,T]$
    \begin{equation} \label{ineq:y-bound}
    \|y\|_{S^\infty(\mathbb{F},\mathbb{R}^n)} \le \left(\norm{\nabla_k g}_\infty + Tu'(\eta_0)\norm{\nabla_k F}_\infty\right)e^{\left(-\delta - \rho + \norm{\nabla_k F}_\infty \right)T}
    \end{equation}
    where $\eta_{0}$ is defined in Lemma \ref{lemma:c-bounded-below}.

    \medskip
    
    \noindent \textit{Step 3: Lispchitz continuity of $\Theta_2$.}
    Let $(k^1,p^1)$ and $(k^2,p^2)$ be in  $L^2(\F,\R^n)\times  L^2(\F,\R^d)$. Let $(y^i,z^i,z^{0,i}) = \Theta_2(k^i,p^i)$ for $i \in \{1,2\}$ and
    \begin{equation*}
        \Delta y = y^1 - y^2, \quad \Delta z = z^1 - z^2, \quad \Delta z^0 = z^{0,1} - z^{0,2}, \quad \Delta p = p^1 - p^2, \quad \Delta k = k^1 - k^2. 
    \end{equation*}
     We have 
    \begin{align*}
        -\dd \Delta y_t = \left(-(\delta+\rho)\Delta y_t + \diag( \hat \sigma \Delta z_t)+ \Delta \upsilon_t \right) \dd t - \Delta z_t \dd W_t - \Delta z^0_t \dd W^0_t,
    \end{align*}
    with terminal condition $\Delta y_T = \nabla_k g(k^1_T,p^1_T) -\nabla_k g(k^2_T,p^2_T)$ and where $\Delta \upsilon_t = \upsilon(k^1_t,p^1_t,y^1_t) - \upsilon(k^2_t,p^2_t,y^2_t)$. 
    Let $\Lambda_t = e^{t \nu}$ for some real parameter $\nu$ (which might be negative) to be determined. 
    Now by the It{\^o} formula we have that  
    \begin{align*}
        \Lambda_t |\Delta y_t|^2 + \int_t^T  \Lambda_s\left( |\Delta z_s|^2 + |\Delta z^0_s|^2 \right) \dd s = \Lambda_T|\Delta y_T|^2 - \int_t^T  \Lambda_s \nu |\Delta y_s|^2  \dd s - 2\int_t^T   \Lambda_s \Delta y_s^\top  \dd \Delta y_s \\
        = |\Delta y_T|^2 - \int_t^T  \Lambda_s \nu |\Delta y_s|^2  \dd s + 2\int_t^T  \Lambda_s \left( -(\delta+\rho)|\Delta y_s|^2 +  \Delta y_s^\top  \diag( \hat \sigma \Delta z_s) + \Delta y_s^\top  \Delta \upsilon_s \right)\dd s - M_t,
    \end{align*}
    where 
    $M_t = - \int_t^T  \Lambda_s \Delta y_s^\top  \Delta z_s  \dd W_s - \int_t^T  \Lambda_s \Delta y_s^\top \Delta z^0_s  \dd W^0_s$ 
    is a martingale term and we recall that $|z|^2 = \Tr( z^{T}z)$ for any $z \in \mathbb{R}^{n \times d}$. 
    By Fenchel's and Cauchy-Schwarz inequalities we have that
    \begin{align*}
        \Delta y_s^\top  \diag( \hat \sigma \Delta z_s) + \Delta y_s^\top  \Delta \upsilon_s  \leq & \frac{1}{2} C_\sigma^2 |\Delta y_s|^2 + \frac{1}{2}|\Delta z_s|^2 + |\Delta y_s||\Delta \upsilon_s| \\
         \leq & \frac{1}{2}(2 C_{\upsilon,y} + C_\sigma^2 + C_{\upsilon,k} +  C_{\upsilon,p}) |\Delta y_s|^2 \\ & + \frac{1}{2}|\Delta z_s|^2 + \frac{1}{2}( C_{\upsilon,k} |\Delta k_s|^2 + C_{\upsilon,p} |\Delta p_s|^2).
    \end{align*}
    Combining with the previous equality yields,
    \begin{align*}
        \Lambda_t |\Delta y_t|^2 \leq \Lambda_T|\Delta y_T|^2 + ( - \nu - 2(\delta + \rho) + 2 C_{\upsilon,y} + C_\sigma^2 + C_{\upsilon,k} +  C_{\upsilon,p}) \int_t^T  \Lambda_s |\Delta y_s|^2 \dd s  \\
        + \int_t^T  \Lambda_s ( C_{\upsilon,k} |\Delta k_s|^2 + C_{\upsilon,p} |\Delta p_s|^2 )\dd s - M_t.
    \end{align*}
    Choosing $\nu = - 2(\delta + \rho) + 2 C_{\upsilon,y} + C_\sigma^2 + C_{\upsilon,k} +  C_{\upsilon,p}$
    the expression simplifies,
    \begin{align*}
        \Lambda_t |\Delta y_t|^2 \leq \Lambda_T|\Delta y_T|^2 
        + \int_t^T  \Lambda_s ( C_{\upsilon,k} |\Delta k_s|^2 + C_{\upsilon,p} |\Delta p_s|^2 )\dd s - M_t.
    \end{align*}
    Recalling that the terminal condition is Lipschitz
    \begin{displaymath}
        |\Delta y_T|^2\le C_{\nabla_k g,k}^2 |\Delta k_T|^2 + C_{\nabla_k g,p}^2 |\Delta p_T|^2,
    \end{displaymath}
    we further estimate
    \begin{align*}
        \Lambda_t |\Delta y_t|^2 \leq \Lambda_T \left(C_{\nabla_k g,k}^2 |\Delta k_T|^2 + C_{\nabla_k g,p}^2 |\Delta p_T|^2\right)
        + \int_t^T  \Lambda_s ( C_{\upsilon,k} |\Delta k_s|^2 + C_{\upsilon,p} |\Delta p_s|^2 )\dd s - M_t.
    \end{align*}
    Dividing by $\Lambda_t$ both sides and taking the expectation yields
    \begin{align*}
        \|\Delta y_t\|^2_{L^2(\mathcal{F}_t,\mathbb{R}^n)} \leq \Lambda_{t,T} \left(C_{\nabla_k g,k}^2 \|\Delta k_T\|^2_{L^2(\mathcal{F}_T,\mathbb{R}^n)} + C_{\nabla_k g,p}^2 \|\Delta p_T\|^2_{L^2(\mathcal{F}_T,\mathbb{R}^d)} \right) \\
        + \int_t^T  \Lambda_{t,s} ( C_{\upsilon,k}^2 \|\Delta k_s\|^2_{L^2(\mathcal{F}_s,\mathbb{R}^n)} + C_{\upsilon,p}^2 \|\Delta p_s\|^2_{L^2(\mathcal{F}_s,\mathbb{R}^d)} )\dd s,
    \end{align*}
    where $\Lambda_{t,s} = \Lambda_{s} \Lambda_{t}^{-1}$. 
    Taking the integral with respect to time both sides, 
    \begin{align*}
        \|\Delta y\|^2_{L^2(\mathbb{F},\mathbb{R}^n)} \leq \left((C_{\nabla_k g,k}^2 + T C_{\upsilon,k})  \|\Delta k\|^2_{L^2(\mathbb{F},\mathbb{R}^n)} + (C_{\nabla_k g,p}^2 + T C_{\upsilon,p})  \|\Delta p\|^2_{L^2(\mathbb{F},\mathbb{R}^n)} \right)\int_0^T  \Lambda_{t,T} \dd t.
    \end{align*}
    Using that $\int_0^T  \Lambda_{t,T} \dd t = \Lambda_{T} \int_0^T  e^{\nu (T-t)} \dd t =  \frac{1}{\nu} \left(e^{\nu T} - 1\right) $, concludes the step and the proof. 
\end{proof}

\bigskip

\begin{proposition}
    \label{prop:capital}
    The mapping $\Theta_3$,
    is well-defined and 
    \begin{displaymath}
        \norm{\Theta_3(p^1,y^1) - \Theta_3(p^2,y^2)}_{L^2(\F)}^2 \le C_4\norm{p^1-p^2}_{L^2(\F)}^2+C_5\norm{y^1 - y^2}_{L^2(\F)}^2,
    \end{displaymath}
    where 
    \begin{align*}
        C_4 & =C_{a,p} T e^{\left(C_{a,p}+C_{a,y}+2C_{a,k}+C_\sigma^2- 2\delta\right) T},\\
        C_5 & =C_{a,y} T e^{\left(C_{a,p}+C_{a,y}+2C_{a,k}+C_\sigma^2- 2\delta\right) T}.
    \end{align*}
\end{proposition}

\begin{proof}
    \textit{Step 1: Well-posedness of $\Theta_3$.}
    For any $(p,y)\in  L^2(\F,\R^d)\times L^2(\F,\R^n)$, the existence and uniqueness of the solution of the third equation of \eqref{main:FBSDE-NE} comes from the Lipschitz properties of the coefficients. In addition, since $\norm{\kappa}_{L^2(\Omega)}<+\infty$, the solution belongs in $ L^2(\F,\R^n)$. Therefore, $\Theta_3$ is well-defined. 

    \medskip
    
    \noindent \textit{Step 2: Lipschitz continuity.}
    Let us verify that the map $\Theta_3$ is Lipschitz.  It is a consequence of the Lipschitz property of $a$. We denote $k^i = \Theta_3(p^i,y^i)$ for $i \in \{1,2\}$ and 
    \begin{equation*}
        \Delta a_t = a(k^1_t,p^1_t,y^1_t) - a(k^2_t,p^2_t,y^2_t).
    \end{equation*}
    Observe that $\Delta k_0 = 0$ and
    \begin{displaymath}
        \dd \Delta k_t =  \left(\Delta a_t - \delta \Delta k_t\right)\dd t + \sigma(\Delta k_t)\dd W_t.
    \end{displaymath}
    Therefore, by It{\^o}'s formula we have 
    \begin{align*}
        \norm{\Delta k_t}^2_{L^2(\mathcal{F}_t,\mathbb{R}^n)} & = \E\left[ \int_0^T  2\Delta k_s^\top  \dd \Delta k_s + \int_0^T  \Tr(\sigma(\Delta k_s)^\top  \sigma(\Delta k_s))ds\right] \\
        & = \E\left[ \int_0^T  2\Delta k_s^\top  \left(\Delta a_t - \delta \Delta k_s\right) \dd s + \int_0^T  \Tr(\sigma(\Delta k_s)^\top  \sigma(\Delta k_s))\dd s\right].
    \end{align*}
    By Fubini's theorem, Cauchy-Schwarz inequality and linearity of $\sigma$, we have 
    \begin{align*}
        \norm{\Delta k_t}^2_{L^2(\mathcal{F}_t,\mathbb{R}^n)} \leq & 2 \int_0^T  \left(\|\Delta k_s \|_{L^2(\mathcal{F}_s,\mathbb{R}^n)} \|\Delta a_s \|_{L^2(\mathcal{F}_s,\mathbb{R}^n)} - \delta \|\Delta k_s \|^2_{L^2(\mathcal{F}_s,\mathbb{R}^n)} \right) \dd s \\ & + C^2_{\sigma} \int_0^T  \|\Delta k_s \|^2_{L^2(\mathcal{F}_s,\mathbb{R}^n)} \dd s.
    \end{align*}
    Now by the Lipschitz property of function $a$, we have
    \begin{align*}
        \norm{\Delta k_t}^2_{L^2(\mathcal{F}_t,\mathbb{R}^n)}  \leq & 2 \int_0^T   \|\Delta k_s \|_{L^2(\mathcal{F}_s,\mathbb{R}^n)} \left( C_{a,p}  \|\Delta p_s \|_{L^2(\mathcal{F}_s,\mathbb{R}^d)} + C_{a,y}  \|\Delta y_s \|_{L^2(\mathcal{F}_s,\mathbb{R}^n)} \right) \dd s \\ & + (2 C_{a,k} + C^2_{\sigma} - 2\delta) \int_0^T  \|\Delta k_s \|^2_{L^2(\mathcal{F}_s,\mathbb{R}^n)} \dd s.
    \end{align*}
    Now the Fenchel inequality yields that 
    \begin{align*}
        \norm{\Delta k_t}^2_{L^2(\mathcal{F}_t,\mathbb{R}^n)}  \leq & \int_0^T  \left( C_{a,p}  \|\Delta p_s \|_{L^2(\mathcal{F}_s,\mathbb{R}^d)}^2 + C_{a,y}  \|\Delta y_s \|_{L^2(\mathcal{F}_s,\mathbb{R}^n)}^2 \right) \dd s \\ & + (C_{a,p}  + C_{a,y}  + 2 C_{a,k} + C^2_{\sigma} - 2 \delta) \int_0^T  \|\Delta k_s \|^2_{L^2(\mathcal{F}_s,\mathbb{R}^n)} \dd s.
    \end{align*}
    By Gr{\"o}nwall's Lemma we have 
    \begin{align*}
        \norm{\Delta k_t}^2_{L^2(\mathcal{F}_t,\mathbb{R}^n)} \le e^{\left(C_{a,p}+C_{a,y}+2C_{a,k}+C_\sigma^2- 2\delta\right) t} \int_0^T  \left(C_{a,p} \|\Delta p_s \|_{L^2(\mathcal{F}_s,\mathbb{R}^d)}^2 + C_{a,y}\|\Delta y_s \|_{L^2(\mathcal{F}_s,\mathbb{R}^n)}^2\right) \dd s.
    \end{align*}
    Taking the integral in time, we deduce that 
    \begin{align*}
        \norm{\Delta k}^2_{L^2(\F,\mathbb{R}^n)} \le T e^{\left(C_{a,p}+C_{a,y}+2C_{a,k}+C_\sigma^2- 2\delta\right) T} \left(C_{a,p} \|\Delta p \|_{L^2(\mathbb{F},\mathbb{R}^d)}^2 + C_{a,y}\|\Delta y \|_{L^2(\mathbb{F},\mathbb{R}^n)}^2\right) . 
    \end{align*}
\end{proof}
\noindent
Under the following condition: 
\begin{equation}
    \label{ass:contraction}
    C_4C_1 + C_5(C_2+C_3C_1) < 1,
\end{equation}
we can state:
\begin{theorem} \label{thm:contraction-mfg}
    There exists a unique equilibrium to the mean field game problem \eqref{Nash-eq} if \eqref{ass:contraction} holds.
\end{theorem}
\begin{proof}
    From Proposition \ref{prop:pollution}, \ref{prop:adjoint} and \ref{prop:capital}, and from direct computations 
    \begin{displaymath}
        \norm{\Theta(k^1) - \Theta(k^2)}_{L^2(\F,\R^n)}\le \left(C_4C_1 + C_5(C_2+C_3C_1)\right)^\frac{1}{2}\norm{k^1 - k^2}_{L^2(\F,\R^n)}.
    \end{displaymath}
    Therefore, a direct application of Picard's fixed-point theorem yields the existence and uniqueness of a fixed-point of $\Theta$ that characterizes a solution $(k,p,y,z,z^0,p)$ of \eqref{main:FBSDE-NE}. In view of section \ref{sec:regu} and Theorem \ref{thm:SMP}, we deduce the existence of an equilibrium $(a,p)$. 

    Suppose that there is another equilibrium $(p',a')$. Note that the following map, denoted $\Psi$, is still a contraction from $L^2(\F,\R^n)$ into $L^2(\F,\R^n)$: let $k\in L^2(\F,\R^n)$, 
    \begin{itemize}
        \item consider $(y,z,z^0)=\Theta_2(k,p')$,
        \item then, associate $\hat k=\Theta_3(p',y)$. 
    \end{itemize}
    Therefore, $\Psi$ has a unique fixed-point. Then, using Theorem \ref{thm:SMP}, we can find $(k',y',z',(z^0)')$ such that, together with $p'$, is a solution to \eqref{main:FBSDE-NE}. It is the unique fixed-point of $\Theta$ and thus coincides with $(k,p,y,z,z^0,p)$. Since, $p=p'$, Theorem \ref{thm:SMP} yields $a'=a$, which yields uniqueness of the equilibrium.
\end{proof}
Condition \eqref{ass:contraction} is primarily theoretical. It identifies parameter regimes under which uniqueness is guaranteed, although the exact constants depend on properties of the optimal control and the utility function $u$ and are not computed explicitly.
We end this section with an interpretation on the constants involved in the contraction condition \eqref{ass:contraction} and then present several situations in which the contraction property occurs.

To improve the discussions readability, we recall the definition of each constant and our notational convention: if $f$ is a Lipschitz function of the form $(x,y) \mapsto f(x,y)$, we denote by $C_{f,x}$ its Lipschitz constant with respect to the variable $x$. We also recall the constants $C_1$ to $C_5$,

\begin{itemize}
    \item $ C_1 = C_{\Phi,e} C_\phi T e^{\left(C_{\Phi,e} + 2 C_{\Phi,p}+C_\gamma^2\right)T}$,
    \item $C_2 =(C_{\nabla_k g,k}^2 + T C_{\upsilon,k}^2) \frac{1}{\nu} \left(e^{\nu T} - 1\right)$,
    \item $C_3 = (C_{\nabla_k g,p}^2 + T C_{\upsilon,p}^2)\frac{1}{\nu} \left(e^{\nu T} - 1\right)$,
    \item $C_4 =C_{a,p} T e^{\left(C_{a,p}+C_{a,y}+2C_{a,k}+C_\sigma^2- 2\delta\right) T}$,
    \item $C_5 = C_{a,y} T e^{\left(C_{a,p}+C_{a,y}+2C_{a,k}+C_\sigma^2- 2\delta\right) T}$,
\end{itemize}
with $\nu = - 2(\delta + \rho) + C_{\upsilon,y}^2 + C_\sigma^2 + C_{\upsilon,k} +  C_{\upsilon,p}$, $C_\gamma = \max_i \gamma_i$ and $C_\sigma = \max_i \sigma_i$.

\paragraph{Interpretation of the constants.}

\noindent \textit{Constant $C_1$.} The constant $C_1$ measures how much the drift of the external variable $\Phi$ depends on the aggregate contribution of countries $e$ and the level of the external variable $p$. In economic terms, it captures the sensitivity of the external variable to collective behaviour. A high $C_1$, driven by large Lipschitz constants $C_{\Phi,e}$ or $C_{\Phi,p}$, indicates that the external variable is highly reactive to countries' decisions which can amplify feedback effects in the system. Moreover, a higher volatility of the common noise increases the constant $C_1$, making it more difficult for the contraction condition to be satisfied.

\medskip

\noindent \textit{Constant $C_2$ and $C_3$.} The constants $C_2$ and $C_3$ describe how changes in countries' states $k$ and the external variable $p$ affect their individual goals. $C_2$ reflects the sensitivity of countries' goals to changes in their own states, while $C_3$ measures how countries' goals respond to the external variable. These constants grow with the time horizon $T$ and the strength of interactions between countries and their environment, highlighting the amplified responses over longer time periods.

\medskip

\noindent \textit{Constant $C_4$ and $C_5$.}  The constants $C_4$ and $C_5$ describe how countries’ optimal control strategies $a$ are influenced by the external variable $p$ and the adjoint variable $y$. Both share a common multiplicative factor:
\begin{displaymath}
T e^{\left(C_{a,p} + C_{a,y} + 2C_{a,k} + C_\sigma^2 - 2\delta\right) T}.
\end{displaymath}
This factor represents the amplification of feedback effects over time, driven by the time horizon $T$, stochastic fluctuations $C_\sigma$, and sensitivity parameters. A higher depreciation rate $\delta$ offsets these effects, stabilising the system. The term $C_{a,p}$ in $C_4$ measures how strongly countries’ controls respond to changes in the external variable $p$, while $C_{a,y}$ in $C_5$ captures the influence of adjoint variable $y$ on countries’ controls. Higher values of $C_{a,p}$ or $C_{a,y}$ indicate stronger coupling between countries’ decisions and the system, increasing the potential for amplified feedback loops.

\bigskip

These constants, when considered together, reflect the interaction between the system's direct and indirect effects. Direct interactions are captured by $C_4C_1$, where countries' decisions directly influence the external variable's dynamics. Indirect feedback effects, represented by $C_5(C_2 + C_3C_1)$, account for how changes in the external variable propagate through countries’ objectives and decisions. The contraction condition holds when the amplification of interactions due to time horizon, sensitivity parameters, and stochastic fluctuations is sufficiently small, or when stabilising factors such as depreciation effectively limit the propagation of feedback in the system.

\paragraph{Discussion on the contraction.}

We can identify three main regimes that helps the contraction to hold. We mean by "help" that a combination of the following regime might lead to contraction. Note that in general these regimes are restrictive. 

\medskip 

\noindent \textit{Small time horizon.}
 As expected, if the time horizon $T$ is small enough, then the contraction condition holds. This is a standard requirements to establish the well-posedness and uniqueness of FBSDEs. 

\medskip

\noindent \textit{Small interaction and large production.} Small interactions helps to get contraction. If the sensitivity with respect to the aggregate contribution of countries of the drift of the external variable is small, meaning that $C_{\phi,e}$ is small, then $C_1$ is small. In addition, if the production is large, we expect that the sensitivity of the marginal production with respect to $k$ to be small. Using the Corollary \ref{cor:bounds_gradient_xi} and its proof, we deduce that $C_{\upsilon,k}$ is small when $\norm{\nabla_kF}_\infty$ and $C_{\nabla_k F,k}$ are small. If moreover $C_{\nabla_k g, k}$ is small, then the constant $C_2$ is small, ensuring that the contraction condition holds.

\medskip

\noindent \textit{Small sensibility of the control by large regularization.} The contraction condition can be satisfied for small enough values of the constants $C_4$ and $C_5$, induced by small enough values of the constants $C_{a,y}$ and $C_{a,p}$. The latter corresponds to low sensibility of the feedback control with respect to the adjoint $y$ and the externality $p$. Recalling that the feedback control is given by 
\begin{equation*}
    a^i (k,p,y) = \exp\left(\frac{1}{\theta}\left(y^i - u'(F(k,p) -  \xi(k,p,y))\right)-1\right),
\end{equation*}
and $y^i$ for each $i \in \{1,\ldots, d\}$ is bounded by Proposition \ref{prop:adjoint}, the Lipschitz constant $C_{a,y}$ can be as small as desired for large values of the regularisation parameter $\theta$. The same reasoning applies for the constant $C_{a,p}$. The function $F$ is Lipschitz with respect to $p$ and by the second inequality of Corollary \ref{cor:bounds_gradient_xi}, $\xi$ is Lipschitz with respect to $p$. Since the consumption $c = F(k,p) -  \xi(k,p,y)$ is bounded from below, the derivative $u'$ is Lipschitz. Then the larger the constant $\theta$, the smaller the constant $C_{a,y}$.

\subsection{Link with the master equation}
\label{subsec:master_equation}
Mean field games can be studied using the master equation. It is a partial differential equation defined on an infinite-dimensional space where its solution should be the value of the game. The well-posedness of this equation motivates the need to develop sharper existence and uniqueness results for the FBSDE system. Indeed, if existence and uniqueness can be ensured for the solutions of the system, then one can define the master field as follows:
\begin{align*}
    &\cU(t,k,p,m)\\
    &=\E\left[\left.\int_t^T \left(u(c_\tau) + K(a(k_\tau,p_\tau, y_\tau)\right)e^{-\rho \tau}\dd \tau + g(k_T,p_T)e^{-\rho (T-t)}\right\vert k_t = k, p_t=p, \mathbb{P}_{\kappa} =m \right],
\end{align*}
with $c_\tau = F(k_\tau,p_\tau)-a(k_\tau,p_\tau, y_\tau) \cdot\ind$ and $(k_\tau,p_\tau,y_\tau)_{\tau\in[t,T]}$ being the unique (strong) solution to 
\begin{equation*}
    \left\{
    \begin{array}{rll}
        \dd p_\tau & = \Phi(\E[\phi(k_\tau)| \cF^0_\tau],p_\tau)\dd \tau + \gamma(p_\tau)\dd W^0_\tau, & p_t = p, \\[0.5em]
        - \dd y_\tau  & =  \nabla_k H(a(k_\tau,p_\tau,y_\tau),k_\tau,p_\tau,y_\tau,z_\tau)  \dd \tau - z_\tau \dd W_\tau - z^0_\tau \dd W^0_\tau, & y_T = \nabla_k g(k_T,p_T), \\[0.5em]
        \dd k_\tau & = (a(k_\tau,p_\tau,y_\tau) - \delta k_\tau) \dd \tau +  k_\tau \sigma \dd W_\tau,  & k_t = \kappa,
    \end{array}
    \right.
\end{equation*}
with $\mathbb{P}_{\kappa}$ the law of $\kappa$ and where $a$ is the optimal investment policy defined in \eqref{def:opt-policy}. If $\cU$ is regular enough, it is expected that it is solution of the following master equation (see \cite[section 5.2]{bertucci2023monotone} or \cite{meynard2024study}):
\begin{align*}
    &\partial_t U + \rho U + \Phi\left(\int \phi dm, p\right) \nabla_p U+\cH(k,p,\nabla_k U) \\
    &-  \frac{1}{2}\trace\left(\sigma(k) \sigma(k)^\top  D^2_k U \right)- \frac{1}{2}\trace\left(\gamma(p) \gamma(p)^\top  D^2_p U \right)\\
    &+\int_{\R_+^n} \partial_m U D_v\cH dm - \frac{1}{2} \trace \left( \gamma(p) \gamma(p)^\top \int_{\R_+^n}\diver_k\left(D_{k}\frac{\delta U}{\delta m}\right) d m \right) = 0,
\end{align*}
with $U(T,k,p,m) = g(k,p)$ for $(t,k,p,m) \in [0,T] \times \mathbb{R}^n_+ \times \mathbb{R}^d_+ \times \mathcal{P}_2(\mathbb{R}^n_+)$. The set $\mathcal{P}_2(\mathbb{R}^n_+)$ denotes the set of probability measures with finite second order moments, defined on $\mathbb{R}^n_+$ and 
\begin{displaymath}
        \cH(k,p, v) = \sup_{a\in\R^n}\left\{u(F(k,p) - \ind\cdot a) + v(a-\delta k) - K(a)\right\}. 
    \end{displaymath}
The derivative $\partial_m$ with respect to the probability measure $m$ is understood in the Lions sens (see \cite{meynard2024study} for a definition, where it is defined under the name ``Wasserstein derivative").

\paragraph{Comment on the framework.}
Under the structural assumptions of the model, the master equation is first-order in the measure variable $m$: the dependence on $m$ appears only through first-order functional derivatives $\delta U/\delta m$ and through integrals of $m$, and no second-order derivatives with respect to the measure appear. See \cite{bertucci2023monotone, bertucci2024noise, meynard2024study} for a rigorous treatment of such first-order structures. In \cite{meynard2024study} the authors obtain existence and uniqueness for a mean-field game problem with one external variable by studying the master equation. In order to prove existence and uniqueness of the master equation in the long run (i.e. for any finite time horizon $T>0$), they consider a monotonous regime, which we encounter under additional assumptions, see paragraph \ref{sec:uniqueness}. 

\subsection{Existence of a weak equilibrium}\label{sec:weak_existence}

To provide the existence of equilibria under more general assumptions and not rely on the contraction condition \eqref{ass:contraction} anymore, we use the notion of weak equilibria introduced in \cite[Definition 2.23, Vol. II]{carmona2018probabilistic}. The main difference with a strong solution is that the probabilistic setup is no longer fixed a priori. 

\begin{definition} \label{def:weak-eq-FBDSE}
    We say that the MFG problem \eqref{Nash-eq} admits a weak equilibrium if we can find a probabilistic setup $(\Omega,\cF, \F, \bP)$ as in \textbf{paragraph
    Stochastic context (S)} (see page \pageref{stochastic-context}) for which there exists a $\F-$adapted continuous process $(k,p)$ such that $\F$ is compatible  with the tuple $(k_0, p_0, W^0, \cL((k,W)\vert\cF^0), W)$ and the variable $(k,p)$ solves, together with some tuple $(y,z,z^0,M)$, the following McKean-Vlasov FBSDE:
    \begin{equation} \label{eq:Mckean-Vlasov}
    \left\{
    \begin{array}{rll}
        - \dd y_t  & =  \nabla_k H(a(k_t,p_t,y_t),k_t,p_t,y_t,z_t)  \dd t - z_t \dd W_t - z^0_t \dd W^0_t - \dd M_t, & y_T = \nabla_k g(k_T,p_T), \\[0.5em]
        \dd k_t & = (a(k_t,p_t,y_t) - \delta k_t) \dd t +  \sigma(k_t) \dd W_t,  & k_0 = \kappa,\\[0.5em]
        \dd p_t & = \Phi(\int_{\R^n}\phi(k) \dd \mu_t(k),p_t)\dd t + \gamma(p_t)\dd W^0_t, & p_0 = p,
    \end{array}
    \right.
    \end{equation}
    where $(M_t)_{t\in[0,T]}$ is a square-integrable càd-lag martingale, with $M_0 = 0$ and zero cross-variation with $(W^0,W)$, and $\mu_t(B,\omega_0) = \cL((k(\omega_0,\cdot), W(\cdot))\circ (\pi^k_t)^{-1}(B)$ where $\pi^k_t$ denotes for any $t \in [0,T]$ the evaluation map on $C([0,T],\R^{n \times n})$ giving the first $n$ coordinates at time $t$, namely $\pi^k_t(\mathbf{k},\mathbf{w}) = \mathbf{k}_t$ for any $(\mathbf{k},\mathbf{w}) \in C([0,T],\R^{n \times n})$ and $B$ a Lebesgue set of $\R^n$.
\end{definition}

We also introduce the weak equilibrium problem associated to system \eqref{eq:Mckean-Vlasov} under an optimal control form: Find a probabilistic setup $(\Omega,\cF, \F, \bP)$ as in \textbf{(S)} for which there exists a $\F-$adapted continuous process $(a,k,p)$ such that $\F$ is compatible with $(k_0, p_0, W^0, \cL(k,W\vert\cF^0), W)$ and $(a,k,p)$ solves
    \begin{equation} \label{argmin:J}
    a \in \argmax_{\alpha \in \mathcal{A}(k,p)} J[p](\alpha),
    \end{equation}
    where 
    \begin{equation} \label{eq:weak-formulation}
    \left\{
    \begin{array}{rll}
        \dd k_t & = (a_t - \delta k_t) \dd t +  \sigma(k_t) \dd W_t,  & k_0 = \kappa,\\[0.5em]
        \dd p_t & = \Phi(\int_{\R^d}\phi(k)d\mu_t(k),p_t)\dd t + \gamma(p_t)\dd W^0_t, & p_0 = p,\\[0.5em]
        \mu_t(\omega_0) & = \cL((k(\omega_0,\cdot), W(\cdot))\circ (\pi^k_t)^{-1}, & \forall \omega_0 \in \Omega_0, \; \dd t \otimes \dd \bP-a.s.
    \end{array}
    \right.
    \end{equation}
Let us denote
\begin{displaymath}
    \bar{A}(k,p) = \{a\in [\underline a, \overline a]^n\,:\, \ind \cdot a \le F(k,p) - \underline c\}
\end{displaymath}
for three constants such that $0< \underline a < \overline a$ and $\underline c >0$. 
    We define the alternative weak equilibrium problem:
    find a tuple $(k,p,a,\mu)$ such that 
    \begin{equation} \label{argmin:J-constrained}
    a \in \argmax_{\alpha \in \bar{\mathcal{A}}(k,p)} J[p](\alpha),
    \end{equation}
    where $(k,p)$ solves \eqref{eq:weak-formulation} and $\bar{\mathcal{A}}(k,p) = \left\{a \in L^\infty(\mathbb{F},\mathbb{R}^n), \; a_t \in \bar{A}(k_t,p_t), \; \mathrm{a.s.}\right\}$.

    \begin{remark} \label{remark:same-points}
         We argue that any tuple $(k,p,a,\mu)$ is a solution to the equilibrium problem  
   \eqref{eq:weak-formulation}- \eqref{argmin:J-constrained} if and only if it is a solution to \eqref{argmin:J}-\eqref{eq:weak-formulation} choosing $\underline a$ low enough, $\overline a$ large enough, and $\underline{c}$ low enough.

    Suppose that there exists a weak MFG equilibrium in the sense of definition \ref{def:weak-eq-FBDSE}. Then by a verification argument (as in Theorem \ref{thm:SMP}) the control $a^*$ defined as
    \begin{equation}
        H(a^*_t,k_t,p_t,y_t,z_t) = \max_{a \in A(k_t,p_t)} H(a,k_t,p_t,y_t,z_t), \quad 0\le t \le T, \quad \dd t \otimes \dd \bP-a.s.,
    \end{equation}
    where we recall that $A(k,p) = \left\{a \in \R_+^n\;:\; \ind\cdot a \le F(k,p)\right\}$,
    is optimal, meaning that $a^* \in \argmax_{a \in \mathcal{A}(k,p)} J[p](a)$. We recall that $a^*$ satisfies
    \begin{equation*}
        a_t^{i,*}  = \exp\left(\frac{1}{\theta}\left(y^i_t - u'(F(k_t,p_t) -  \xi(k_t,p_t,y_t))\right)-1\right),
    \end{equation*}
    for each $i \in \{1,\ldots,n\}$. From the $S^\infty$ estimate of $y$ established in Proposition \ref{prop:adjoint} and Lemma \ref{lemma:c-bounded-below}, there exists a constant $\underline c$, such that $F(k_t,p_t) -  \xi(k_t,p_t,y_t) \geq \underline{c}$.\\
    From the expression of $a^*$, the previous inequality and the bound on $y$ (that still holds using the same argument developed in the proof of Proposition \ref{prop:adjoint}) imply that there exists $\underline a >0$ such that $a_t^{i,*} \geq \underline a$ almost surely for every $i \in \{1,\ldots,n\}$.\\
    For the bound by above, we use that $u'$ is positive and $y$ is bounded to deduce that there exists $\overline a >0$ such that $a_t^{i,*} \leq \overline a$ almost surely for every $i \in \{1,\ldots,n\}$.\\
    By fixing such $\underline c$, $\underline a$ and $\overline{a}$, any tuple $(k,p,a,\mu)$ is solution to the weak equilibrium problem \eqref{eq:weak-formulation}-\eqref{argmin:J-constrained} if and only if $(k,p,a,\mu)$ is a solution to the weak equilibrium problem \eqref{argmin:J}-\eqref{eq:weak-formulation}, since the constraint is not binding in the last optimisation problem.
    \end{remark}

   Before proving the main result of the section, we adapt the following result, Theorem 1.60 in \cite[ Vol. II]{carmona2018probabilistic}, to our framework.
\begin{theorem} \label{thm:decoupling-field}
    Given an admissible set-up $(\Omega, \cF, \F, \bP)$ and some inputs $(k_0,p_0, W^0, \mu, W)$. The forward-backward system \eqref{eq:Mckean-Vlasov} has a unique solution $(k,p,y,z,z^0,M)$ such that 
    \begin{equation}
        \E\left[\sup_{t\in[0,T]} \left(\abs{k_t}^2 + \abs{p_t}^2 + \abs{y_t}^2+\abs{M_t}^2\right) + \int_0^T  \left(\abs{z_t}^2 + \abs{z^0_t}^2 \right)\dd t\right]< +\infty,
    \end{equation}
    and if we set $a^* = (a(k_t,p_t,y_t))_{t\in[0,T]}$, then for any admissible control $a \in \mathcal{A}(k,p)$, it holds:
    \begin{displaymath}
        J[p](a^*) - \lambda \mathbb{E}\left[ \int_0^T \abs{a_t^* - a_t}^2 \dd t \right] \ge J[p](a),
    \end{displaymath}
    for some $\lambda >0$.
    The FBSDE \eqref{eq:Mckean-Vlasov} admits a decoupling field $U$ which is $C-$Lipschitz-continuous in $(k,p)$ uniformly in the other variables, for a constant $C$ which only depends on the Lipschitz constants of the data and $T$, and in particular, which is independent of $t\in[0,T]$. As a result, $y$ can be represented as a function of $(k,p)$ as in Proposition 1.50 from \cite[Vol. II]{carmona2018probabilistic}.
\end{theorem}

\begin{proof}

    Let us note minor differences in the statement of Theorem \ref{thm:decoupling-field} and  Theorem 1.60 from \cite{carmona2018probabilistic}. First, there is an external process $p$, $\F^0-$adapted such that its law only depends on $\mu$. Second, the control, the state and the external variable belong to
    \begin{equation*}
        \mathcal B \coloneqq \left\{(a,k,p) \in\R_+^n  \times \R_+^n ,\;  a \in \bar{A}(k,p), \; k\in\R_+^n \right\}.
    \end{equation*}
    Such changes do not break down the proof of Theorem 1.60 in \cite{carmona2018probabilistic}. The running cost
    \begin{displaymath}
    f:(a,k,p) \mapsto -u(F(k,p) - \ind \cdot a) + \theta \sum_{i=1}^n a^i\ln(a^i),
    \end{displaymath}
     is Lipschitz on $\mathcal B$. Also, the first term is convex in $(a,k)$ and the adjustment costs term is strongly convex on $[\underline a, \overline{a}]$, therefore there exists $\lambda>0$ such that for any $(a,a',k,k',p)$ such that $(a,k,p)$ and $(a',k',p)$ both belongs to $\mathcal B$, 
    \begin{equation} \label{ineq:f-strong-convex}
        f(a,k,p) - f(a',k',p)  - \nabla_{(a,k)} f(a',k',p)\cdot(a-a',k-k') \ge  \lambda \abs{a - a'}^2.
    \end{equation}
    Then, following the same proof as in \cite[Vol. II]{carmona2018probabilistic}, for the constrained problem \eqref{eq:weak-formulation}-\eqref{argmin:J-constrained}, we obtain the results for the original problem (see Remark \ref{remark:same-points}).    
\end{proof}

    \begin{theorem} \label{thm:weak-existence}
        If the standing assumptions of the paper hold, and if furthermore, $\Phi$ is continuously differentiable with respect to $p$, then the MFG problem admits a weak equilibrium in the sense of Definition \ref{def:weak-eq-FBDSE}. 
    \end{theorem}

\begin{proof}
    We apply Theorem 3.1 from \cite[Vol. II]{carmona2018probabilistic}. The proof consists in verifying that all the assumptions of the theorem holds for the alternative (or constrained) weak equilibrium problem \eqref{eq:weak-formulation}-\eqref{argmin:J-constrained} and using Remark \ref{remark:same-points}.

    \medskip
    
    \noindent \textit{Step 1.} We verify that the assumptions ``Coefficients MFG with a common noise'' \cite[p. 158-159, Vol. II]{carmona2018probabilistic} are satisfied. The drift 
    \begin{equation*}
        b(k,p,a,\mu) = b_1(k,p,\mu) + b_2 a,
    \end{equation*}
    of the couple process $(k,p)$ is indeed of separable form, where
    \begin{equation*}
        b_1(k,p,\mu) = \left(- \delta k , \Phi\left(\int \phi \dd \mu  , p\right)\right), \quad b_2 = \left( \begin{array}{cc}
             1 &0  \\
             0 & 0
        \end{array}\right).
    \end{equation*}
    verifying Assumption (A1).
    The growth assumption (A2) on the drift and the volatility are clearly satisfied. The property (A3) comes from the Lipschitz regularity of the coefficients, and (A4) from the semi-convexity of the running cost as developed in the proof of Theorem \ref{thm:decoupling-field}.

    \medskip

\noindent \textit{Step 2.} We now verify the Assumptions ``FBSDE MFG with a Common Noise''\cite[p. 159-161, Vol. II]{carmona2018probabilistic}. The Assumption (A1) and (A3) are verified by the verification argument developed in Remark \ref{remark:same-points}. 
By Theorem \ref{thm:decoupling-field} the FBSDE admits a decoupling field $U$ which is Lipschitz implying that the assumption ``Iteration in Random Environment'' in \cite[p. 67-68, Vol. II]{carmona2018probabilistic} holds and that Theorem 1.53 can be applied. This implies that (A2) holds. Finally the regularity assumptions (A4)-(A6) are consequences of the Lipschitz property of the feedback control $a(\cdot)$ (see Proposition \ref{prop:a-Hk-Lip-maps}), and the mappings $F$, $g$, $\phi$ and $\Phi$.
\end{proof}

\subsection{Uniqueness under a \texorpdfstring{$L^2$}{L2}-monotonicity condition \label{sec:uniqueness}}

In this section, we show how to obtain uniqueness of strong equilibria without contraction.
In the spirit of \cite{bertucci2024noise, meynard2024study}, we introduce uniqueness under a $L^2-$monotonicity condition. By ease of notation, we set
\begin{displaymath}
    \cH(k,p,y)= \sup_{a\in A(k,p)}\left(a\cdot y + u(F(k,p) - a\cdot \ind) - \theta\sum_{i=1}^n a^i\ln(a^i)\right),
\end{displaymath}
where we recall that $A(k,p) = \left\{a \in \R_+^n\;:\; \ind\cdot a \le F(k,p)\right\}$, 
and 
\begin{displaymath}
\hat\Phi(\mu,p)= \Phi\left(\int_{\mathbb{R}^n}\phi(k')d\mu(k'),p\right).
\end{displaymath}
The conditions under which uniqueness holds are: 
\begin{assumption}
    \label{ass:uniqueness}
    There exists a matrix $A$ of size $d\times d$ and a matrix $B$ of size $n \times n$, symmetric and positive definite such that for any $(k,p,y,k',p',y')\in (L^2(\Omega,\R^n)\times \R^d\times L^2(\Omega,\R^n))^2$
\begin{equation}\label{eq:monotonicity}
\begin{split}
  \E\Bigl[ &\left(\nabla_y\cH(k,p,y)-\nabla_y\cH(k',p',y')\right)\cdot\Delta y_t \\[0.5em] &- \left(\nabla_k\cH(k,p,y)-\nabla_k\cH(k',p',y')\right)\cdot\Delta k_t\\[0.5em]
        & - \Delta p^\top A  (\hat\Phi(\mu,p) - \hat\Phi(\mu',p'))+ \frac{\rho}{2} \Delta p^\top A \Delta p - \frac{1}{2}\trace(\gamma(\Delta p)^\top A \gamma(\Delta p)) \\[0.5em]
         & + \left(\frac{\rho}{2} + \delta \right)\Delta k^\top B \Delta k  - \Delta k^\top B \left(\nabla_y\cH(k,p,y)-\nabla_y\cH(k',p',y')\right)  \\[0.5em]
         &-\frac{1}{2}\trace(\sigma(\Delta k)^\top B \sigma(\Delta k) )
  \Bigr]\;\geq \;0.
\end{split}
\end{equation}
where the inequality is strict when $(k,p,y)\neq (k',p',y')$, 
and 
\begin{equation}
    \label{eq:monotonicity_2}
    \E\left[\Delta k \cdot(\nabla_kg(k,p)-\nabla_kg(k',p')) - \frac{1}{2} \Delta p^\top A\Delta p - \frac{1}{2}\Delta k^\top B\Delta k \right]\le 0.
\end{equation}
In the above, we set
\begin{equation*}
   \Delta k = k - k',
  \quad \Delta p = p - p',
  \quad \Delta y = y - y', \quad \Delta z = z - z'. 
\end{equation*}
\end{assumption}

\begin{remark}
    We can observe that a high common noise volatility makes the condition \eqref{eq:monotonicity} harder to be satisfied.
\end{remark}

\begin{theorem}
    \label{lemma:strong_uniqueness}
    Under Assumption \ref{ass:uniqueness}, there exists at most one strong equilibrium to \eqref{Nash-eq}.
\end{theorem}
\begin{proof}
    By contradiction, let us consider two distinct strong equilibria $(k,p,y,z,z^0)$ and\linebreak $(k',p',y',z',(z^0)')$, and set
\begin{equation*}
   \Delta k_t = k_t - k'_t,
  \quad \Delta p_t = p_t - p'_t,
  \quad \Delta y_t = y_t - y'_t, \quad \Delta z_t = z_t - z'_t. 
\end{equation*}
Define the conditional laws
\begin{equation*}
  \mu_t = \mathcal{L}(k_t \vert \mathcal{F}^0_t),
  \quad
  \mu'_t = \mathcal{L}(k'_t \vert \mathcal{F}^0_t).
\end{equation*}
Consider the quantity
\begin{equation} \label{eq:def-J}
  J := \E\left[e^{-\rho T}\Delta k_T \cdot \Delta y_T
      - e^{-\rho T} \frac{1}{2}( \Delta p_T^{\top} A \Delta p_T + \Delta k_T^\top B \Delta k_T) \right].
\end{equation}
Ito's formula and the fact that the pair $(k,k')$ and $(p,p')$ have the same initial condition yield
\begin{align*}
        J = & \E\left[\int_0^Te^{-\rho t}\left( \Delta k_t\cdot \dd \Delta y_t + \Delta y_t \cdot  \dd \Delta k_t - \Delta p_t^\top A \dd \Delta p_t - \Delta k_t^\top B \dd \Delta k_t\right)\right]\\
         & + \E\left[\int_0^T e^{-\rho t}\left(- \rho\Delta k_t \cdot \Delta y_t + \frac{\rho}{2}\Delta p_t^\top A \Delta p_t + \frac{\rho}{2}\Delta k_t^\top B \Delta k_t\right) \dd t \right] \\
        & +\E\left[\int_0^T e^{-\rho t}\left(\trace(\sigma(\Delta k_t)^\top  \Delta z_t) - \frac{1}{2} \trace(\gamma(\Delta p_t)^\top A\gamma(\Delta p_t)) -\frac{1}{2} \trace(\sigma(\Delta k_t)^T B\sigma(\Delta k_t) \right) \dd t \right].
    \end{align*}
    Recalling the definition of $(\hat \sigma_{i,j}= \sigma_i \delta_{i,j})_{i,j \in \{0,\ldots,n\}}$ we have that $D_k \trace(\sigma^\top(k_t) z_t) =  \diag(\hat{\sigma}^\top z_t)$. 
    Moreover,
    \begin{align*}
       - \dd \Delta y_t& = \left[-(\delta + \rho)\Delta y_t - \diag(\hat{\sigma}^\top z_t) + \diag(\hat{\sigma}^\top z_t') + \nabla_k\cH(k_t,p_t,y_t) - \nabla_k\cH(k_t',p_t',y_t')\right] \dd t \\
        &\quad\quad- \Delta z_t dW_t - \Delta z^0_t \dd W_t^0. 
    \end{align*}
    Then, using this equation together with the fact that 
    \begin{equation*}
        \diag(\hat{\sigma}^\top \Delta z_t) \cdot \Delta k_t = D_k \trace(\sigma^\top(\Delta k_t) \Delta z_t)\cdot \Delta k_t = \trace(\sigma^\top(\Delta k_t) \Delta z_t),
    \end{equation*}
    and the definition of $\cH$ leads to 
    \begin{align*}
        J = & \E\left[\int_0^Te^{-\rho t}\left(\nabla_y\cH(k_t,p_t,y_t)-\nabla_y\cH(k_t',p_t',y_t') - \delta\Delta k_t\right)\cdot\Delta y_t \dd t \right]\\ 
         & - \E\left[\int_0^Te^{-\rho t}\left(\nabla_k\cH(k_t,p_t,y_t)-\nabla_k\cH(k_t',p_t',y_t') -  (\rho + \delta)\Delta y_t\right)\cdot\Delta k_t \dd t\right]\\
        &-\E\left[\int_0^T\rho e^{-\rho t}\Delta k_t\cdot \Delta p_t dt\right] \\
        &-\E\left[\int_0^T e^{-\rho t} \Delta p_t^\top  A (\hat\Phi(\mu_t,p_t) - \hat\Phi(\mu_t',p_t'))  \dd t\right] \\
        & +\E\left[\int_0^T e^{-\rho t}\left( \frac{\rho}{2} \Delta p_t^\top A \Delta p_t + \left(\frac{\rho}{2} + \delta\right) \Delta k_t^T B \Delta k_t \right) \dd t\right] \\
        & - \E\left[\int_0^T e^{-\rho t}\Delta k_t^T B \left(\nabla_y\cH(k_t,p_t,y_t)-\nabla_y\cH(k_t',p_t',y_t')\right) \dd t \right]\\
        &  -\frac{1}{2} \E\left[\int_0^T\left( \trace(\gamma(\Delta p_t)^\top A \gamma(\Delta p_t)) + \trace(\sigma(\Delta k_t)^\top B \sigma(\Delta k_t) \right) \dd t\right].
    \end{align*}
    Therefore, using Fubini's theorem and the law of iterated expectation, i.e. for any integrable random variable $X$, $\E[X] = \E[\E[X\vert\cF_t^0]]$, we obtain 
    \begin{align*}
        J =  \E\bigg[  \int_0^T e^{-\rho t}\E\bigg[  & \left( \nabla_y\cH(k_t,p_t,y_t)-\nabla_y\cH(k_t',p_t',y_t')\right)\cdot\Delta y_t \\[0.5em] &- \left(\nabla_k\cH(k_t,p_t,y_t)-\nabla_k\cH(k_t',p_t',y_t')\right)\cdot\Delta k_t\\[0.5em]
        & - \Delta p_t^\top A  (\hat\Phi(\mu_t,p_t) - \hat\Phi(\mu_t',p_t'))+ \frac{\rho}{2} \Delta p_t^\top A \Delta p_t + \left(\frac{\rho}{2}+ \delta \right)\Delta k_t^\top B \Delta k_t   \\[0.5em]
        &- \Delta k_t^\top B \left(\nabla_y\cH(k_t,p_t,y_t)-\nabla_y\cH(k_t',p_t',y_t')\right)
        \\[0.5em]
         & - \frac{1}{2}\trace(\gamma(\Delta p_t)^\top A \gamma(\Delta p_t)) - \frac{1}{2}\trace(\sigma(\Delta k_t)^\top B \sigma(\Delta k_t) \bigg\vert\cF^0_t\bigg] \dd t \bigg].
    \end{align*}
    On the one hand, the monotonicity  assumption  \eqref{eq:monotonicity} ensures that $J$ is positive whenever $(k,p,y) \neq (k',p',y')$. On the other hand, the terminal condition's monotonicity assumption  \eqref{eq:monotonicity_2}, in conjunction with the definition  \eqref{eq:def-J} also implies that $J$ is non-positive. This contradiction completes the proof.
\end{proof}

To translate condition \eqref{eq:monotonicity} into a constraint on the primitive data, it is convenient to view $\cH$ as the Legendre–Fenchel transform of
\begin{displaymath}
    f(a,k,p) = -u(F(k,p) - \ind\cdot a) +\theta\sum_{i=1}^n a^i\ln(a^i).
\end{displaymath}
   Mimicking the classical convex‐analysis argument in \cite{Meszaros2024}, where the authors show the equivalence of their Eq. (2.8) and (H8), we arrive at:
\begin{lemma}
    \label{lem:equivalence_conditions}
    For any $(k,p,a,k',p',a')\in (L^2(\Omega,\R^n)\times \R^d\times L^2(\Omega,\R^n))^2 $, it holds
   \begin{equation}
    \label{eq:monotonicity_running_cost}
    \begin{split}
    \E\big[
        & \left(\nabla_a f (a,k,p) - \nabla_a  f(a',k',p') \right) \cdot \Delta a\\[0.5em]  
        & + \left(\nabla_k f(a,k,p) -\nabla_kf(a',k',p') \right) \cdot \Delta k \\[0.5em]
        &  - \Delta p^\top A  (\hat\Phi(\mu,p) - \hat\Phi(\mu',p'))+ \frac{\rho}{2} \Delta p^\top A \Delta p + \left(\frac{\rho}{2}+ \delta \right)\Delta k^\top B \Delta k  - \Delta k^\top B \Delta a \\[0.5em]
         & - \frac{1}{2} \trace(\gamma(\Delta p)^\top A \gamma(\Delta p)) - \frac{1}{2} \trace(\sigma(\Delta k)^\top B \sigma(\Delta k)  \big] \ge 0, 
         \end{split}
\end{equation}
where $ \Delta k = k - k'$, $ \Delta p = p - p'$, $ \Delta a = a - a'$. The inequality is strict when $(k,p,y)\neq(k',p',y')$, if and only if \eqref{eq:monotonicity} holds with 
\begin{displaymath}
    \cH(k,p,y) = \sup_{a\in A(k,p)}\{ay - f(a,k,p)\}.
\end{displaymath}
\end{lemma}
\begin{proof}
We give the proof for completeness, we borrow it from \cite{Meszaros2024}.
We recall that the Hamiltonian $\cH$ and the running cost $f$ are Legendre-Fenchel conjugates of each other with respect to  the control variable:
\begin{displaymath}
\cH(k,p,y) \;=\; \sup_{a\in A(k,p)}\{a\cdot y - f(a,k,p)\},
\qquad
f(a,k,p) \;=\; \sup_{y\in\R^n}\{a\cdot y - \cH(k,p,y)\}.
\end{displaymath}
By the Fenchel identity, we have the equality $\cH(k,p,y) + f(a,k,p) = a\cdot y$ if and only if
\begin{displaymath}
    a = \nabla_y \cH(k,p,y)
\;\;\text{or equivalently}\;\;
y = \nabla_a f(a,k,p).
\end{displaymath}
Moreover, 
\begin{displaymath}
\nabla_k \cH(k,p,y)
= -\,\nabla_k f\bigl(\nabla_p \cH(k,p,y),k,p\bigr),\qquad
\nabla_k f(a,k,p)
= -\nabla_k \cH(k,p,\nabla_a f(a,k,p)).
\end{displaymath}

\medskip

Assume \eqref{eq:monotonicity} and fix any $(k,p,a,k',p',a')\in (L^2(\Omega,\R^n)\times \R^d\times L^2(\Omega,\R^n))^2 $. Defining $y = \nabla_af(a,k,p)$ and $y' = \nabla_af(a',k',p')$. It implies that $a=\nabla_y\cH(k,p,y)$ and $a' =\nabla_y\cH(k',p',y')$. Then, we observe that the first term of \eqref{eq:monotonicity_running_cost} is exactly the first term of \eqref{eq:monotonicity}. The second term and the term $\Delta k^\top B \Delta a$ are treated similarly, leading to \eqref{eq:monotonicity_running_cost}. The converse implication can be checked in a similar way. 
\end{proof}

\begin{remark}
    Condition \eqref{eq:monotonicity_2} and \eqref{eq:monotonicity_running_cost} are stronger than needed. In section \ref{sec:weak_existence}, we observed that at equilibrium $(a,k,p)$ takes value in $\mathcal{B}$ a.s. (defined in the proof of Theorem \ref{thm:decoupling-field}). Therefore, \eqref{eq:monotonicity_2} and \eqref{eq:monotonicity_running_cost} can be stated only for $(a,k,p)$ and $(a',k',p',)$ in $L^2(\Omega,\R^n)\times \R^d\times L^2(\Omega,\R^n)$ and taking value in $\mathcal{B}$ a.s..
\end{remark}

Only the implication is needed to establish:
\begin{corollary}
    If $\Phi$ is continuously differentiable with respect to $p$, then under the usual assumptions completed with  \eqref{eq:monotonicity_2} and \eqref{eq:monotonicity_running_cost}, there exists a unique strong MFG equilibrium to \eqref{Nash-eq}.
\end{corollary}  

\begin{proof}
We use Theorem 2.29 from \cite[Vol. II]{carmona2018probabilistic} to state that strong uniqueness and weak existence of an equilibrium yields the existence of a unique strong solution. 
\end{proof}

In the example below, we develop a natural model where the drift of the externality $p$ is linear. The others functions of the model are assumed to satisfy the standing assumptions of the paper. In this framework, we find conditions ensuring the monotone regime to hold, yielding existence and uniqueness of a strong equilibrium.
\paragraph{Example.}
Let us now examine an example where the uniqueness of a mean-field game equilibrium is ensured by verifying that conditions \eqref{eq:monotonicity_2} and  \eqref{eq:monotonicity_running_cost} hold.

\medskip
\noindent \textit{Framework.} Fix $d=1$ and $n=2$. Suppose
\begin{equation} \label{eq:phi-example}
    \hat\Phi(\mu,p) = C_{\Phi,e}\int_{\mathbb{R}_+^2} k^1 \dd \mu(k^1,k^2) - C_{\Phi,p}p,
\end{equation}
and assume that 
\begin{equation*}
    \sigma(k) = \sigma \left( \begin{array}{cc}
        k^1 & 0 \\
        0 & k^2
    \end{array}\right), \quad \gamma(p) = \gamma p,
\end{equation*}
with $\sigma, \gamma \in \mathbb{R}_+^*$. Assume that the terminal cost is linear, i.e. $g(k,p) =  \sum_{i=1}^n k^i- p$.

\medskip
\noindent \textit{Conditions.} Conditions under which strong existence and uniqueness holds are the following:
there exists $\alpha,\;\beta>0$ and $\varepsilon\in[0,1]$ such that
\begin{eqnarray*}
2 \lambda &>& (1-\varepsilon)\beta + \frac{1}{2}\norm{D_{a,p}f}_\infty,\\
\beta\left(\frac{\rho}{2}+\delta -\frac{\sigma^2}{2} -\varepsilon\right)&>&\frac{1}{2}\norm{D_{k,p}f}_\infty + \alpha C_{\Phi,e},\\ \alpha\left(C_{\Phi,p}- \frac{C_{\Phi,e}}{2} + \delta - \frac{\gamma^2}{2} \right)&>&\frac{1}{2}\norm{D_{a,p}f}_\infty+\frac{1}{2}\norm{D_{k,p}f}_\infty.
\end{eqnarray*}

\medskip
\noindent \textit{Justification.}
Let $\alpha, \beta >0$. We consider two matrix $A = \alpha I_d$ and $B = \beta I_n$.
Under the above specification, the conditions 
\eqref{eq:monotonicity_2} and \eqref{eq:monotonicity_running_cost}  simply writes
\begin{equation}  \label{eq:sufficient-cond-hamilton-var-ex}
    \begin{split}
    \E\big[
        & \left(\nabla_k f(a,k,p) -\nabla_kf(a',k',p') \right) \cdot \Delta k +\left(\nabla_a f (a,k,p) - \nabla_a  f(a',k',p') \right) \cdot \Delta a\\[0.5em]
        & -\alpha C_{\Phi,e} \Delta k^1  \Delta p + \alpha \left(C_{\Phi,p} + \frac{\rho}{2} - \frac{\gamma^2}{2} \right)|\Delta p|^2 +\beta\left( \delta + \frac{\rho}{2} - \frac{\sigma^2}{2}\right)\abs{\Delta k}^2-\beta\Delta k\cdot \Delta a\big] \ge 0,
    \end{split}
\end{equation}
where the inequality is strict whenever $(a,k,p)\neq (a',k',p')$ and 
\begin{equation} \label{eq:sufficient-cond-term-var-ex}
     -\beta \E[|\Delta k|^2] - \alpha \E[|\Delta p|^2] \leq 0.
\end{equation}
It is clear that condition \eqref{eq:sufficient-cond-term-var-ex} is always satisfied. Then we turn to \eqref{eq:sufficient-cond-hamilton-var-ex}.
Using \eqref{ineq:f-strong-convex}, we have that 
\begin{align*}
    & \mathbb{E}\left[\left(\nabla_k f(a,k,p) -\nabla_k f(a',k',p') \right) \cdot \Delta k +\left(\nabla_a f (a,k,p) - \nabla_a  f(a',k',p') \right) \cdot \Delta a\right] \\
     & \ge  \mathbb{E}\left[2 \lambda \abs{\Delta a}^2 +\left(\nabla_k f(a',k',p) -\nabla_kf(a',k',p') \right) \cdot \Delta k  +\left(\nabla_a f (a',k',p) - \nabla_a  f(a',k',p') \right) \cdot \Delta a \right]. 
\end{align*}
As a consequence, a sufficient condition for \eqref{eq:monotonicity_running_cost} is given by
\begin{equation}\label{ineq:monotony-example}
\begin{split}
     \mathbb{E}\bigg[ & 2 \lambda \abs{\Delta a}^2 + \alpha  \left(C_{\Phi,p} + \frac{\rho}{2} - \frac{\gamma^2}{2} \right)|\Delta p|^2 +\beta\left( \delta + \frac{\rho}{2} - \frac{\sigma^2}{2}\right)\abs{\Delta k}^2 \\
    & + \left(\nabla_k f(a',k',p) -\nabla_k f(a',k',p') \right) \cdot \Delta k\\[0.75em]
    & +\left(\nabla_a f (a',k',p) - \nabla_a  f(a',k',p') \right) \cdot \Delta a \\
    & -\alpha C_{\Phi,e}\left(k^1 - (k^1)' \right) \Delta p -\beta\Delta k\cdot \Delta a \bigg]\ge 0.
\end{split}
\end{equation}
By the fundamental theorem of calculus, we have
\begin{align*}
    \nabla_k f(a',k',p) -\nabla_k f(a',k',p') &= \int_0^1 D_{k,p}f(a^\theta,k^\theta,p^\theta)\dd \theta\Delta p,\\
    \nabla_a f(a',k',p) -\nabla_a f(a',k',p') &= \int_0^1 D_{a,p}f(a^\theta,k^\theta,p^\theta) \dd \theta  \Delta p,
\end{align*}
where $a^\theta = (1-\theta)a' + \theta a$, $k^\theta = (1-\theta)k' + \theta k$, $p^\theta = (1-\theta)p' + \theta p$.
We observe that the left hand-side of the inequality \eqref{ineq:monotony-example} is
\begin{displaymath}
    \mathbb{E}\left[(\Delta a^\top,\Delta k^\top,\Delta p) M (\Delta a^\top,\Delta k^\top,\Delta p)^\top\right],
\end{displaymath}
where the matrix
\begin{displaymath}
    M = \left(\begin{array}{ccc}
         2\lambda I_n &-(1-\varepsilon)\beta I_n&Q \\
         -\varepsilon\beta I_n&\beta\left( \delta + \frac{\rho}{2} - \frac{\sigma^2}{2}\right)I_n&R(\alpha) \\
        Q^\top& R^\top(\alpha)& \alpha \left(C_{\Phi,p} + \frac{\rho}{2} - \frac{\gamma^2}{2} \right)
    \end{array}\right)
\end{displaymath}
for some $\varepsilon \in [0,1]$, and the sub-matrix $Q$ and $R(\alpha)$ are given by 
\begin{displaymath}
    Q = \frac{1}{2} \int_0^1 D_{a,p}f(a_s,k_s,p_s) \dd s, \quad R(\alpha) =\frac{1}{2}\int_0^1 D_{k,p}f(a_s,k_s,p_s) \dd s  + \frac{1}{2}\left(
        -\alpha C_{\Phi,e}, 0
   \right)^\top.
\end{displaymath}
The monotony condition is now reformulated as a positive-definiteness condition for the matrix $M$.
A precise way to determine the positiveness of $M$ is to study the positiveness of its principal minors. However, to avoid complex calculations and keep the presentation simple, we restrict the condition by asking $M$ to be diagonally dominant. It is the case when the condition mentioned in the beginning of the example are satisfied.

\medskip
\noindent
\textit{On a bound by below of $\lambda$.}
We finally provide sufficient conditions on the model's data to ensure that the coefficient $\lambda$ is large enough, that is to say greater than a fixed constant $c>0$.
Essentially, $\lambda$ is the lowest value of the Hessian matrix of $f$ with respect to the investment variable $a$. For any $(a,k,p) \in \mathcal{B}$, the second order derivative of the running cost satisfies
\begin{align*}
    \lambda = \min_{i \in \{1,\ldots,n\}} (D^2_a f(a,k,p))_{i,i} \geq \theta \min_{i \in \{1,\ldots,n\}} \frac{1}{a^i}.
\end{align*}
Recalling that $a(\cdot)$ is in $\bar{A}(k,p)$ and satisfies \eqref{eq:opt-policy-ai} and the bound \eqref{ineq:y-bound} on $y$, we have that 
\begin{align*} 
    \frac{1}{\theta} a^i(k,p,y) & \leq \frac{1}{\theta} \exp\left(\frac{1}{\theta}\|y\|_{S^{\infty}(\mathbb{F},\mathbb{R}^n)}-1\right), \\[0.5em]
    &\leq \frac{1}{\theta} \exp\left(\frac{1}{\theta}\left(\norm{\nabla_k g}_\infty + Tu'(\eta_0)\norm{\nabla_k F}_\infty\right)e^{\left(-\delta - \rho + \norm{\nabla_k F}_\infty \right)T}-1\right) \\[0.5em]
    & \leq \frac{1}{c},
\end{align*}
for any $i \in \{1,\ldots,n\}$, where we assumed, i.e. impose to the model, that 
\begin{equation*}
    \norm{\nabla_k g}_\infty = 1, \quad \norm{\nabla_k F}_\infty \leq \delta + \rho, \quad c \leq \theta \exp\left(-\frac{1}{\theta} Tu'(\eta_0)\norm{\nabla_k F}_\infty \right).
\end{equation*}
Therefore $\theta / a^i(k,p,y) \geq c$ and taking the minimum with respect to the index $i \in \{0,\ldots,n\}$ yields that 
\begin{equation*}
    \lambda \geq \theta \min_{i \in \{0,\ldots,n\}} \frac{1}{a^i}  \geq c, 
\end{equation*}
as desired.

\section{Numerical simulations} \label{sec:num-sim}

This section is dedicated to numerical resolution. We describe the numerical method in Section \ref{sec:numerical-method}. We specify the model in Section \ref{sec:num-result}. This example is concerned with an economy with two types of capital: a brown capital with high productivity but high exposure to climate change and a green capital with low productivity but insensible to climate change. We discuss the numerical results at the end of the section. 

\subsection{Algorithm} \label{sec:numerical-method}
To solve Problem \eqref{Nash-eq}, we rely on a fixed-point method that fixes the pollution $p$:
Once $p$ is fixed, the problem is still in high dimension and we use a neural network to approximate the optimal control.
Then it is possible to solve the Pontryagin optimality equations given by \eqref{main:FBSDE-NE} by discretizing the problem in time and adapting to the mean field case one of the algorithms developed in \cite{germain2022numerical}.\\
We do not follow this approach, but use a direct one as proposed in \cite{carmona2022convergence}: 
Let $N_T$ be a positive integer, let $\Delta_t =\frac{T}{N_T}$ and $t_n= n \Delta t, n=0, \ldots, N_T$.
At each fixed-point iteration $j$ of the algorithm, assuming we have an approximation $R^{j-1}_t$ of $\mathbb{E}[\phi(k_t)| \cF^0_t]$, the computation consists of two parts:
\begin{enumerate}
    \item 
    First, we  approximate the control at each time $t_n$ by a single feedforward network 
    \begin{equation*}
        a^\xi :  [0,T] \times \R^d \times \R^n \longrightarrow \R_+^n
    \end{equation*}
     with parameters $\xi$, taking time as input as in \cite{chan2019machine}. Then we solve:
\begin{align*}
    \xi_j^* = \argmax_\xi  & \Delta t\sum_{i=0}^{N_T-1} \mathbb{E}\left[  \big( u( F(k^{j}_{t_i},p_{t_i}^j) - \mathds{1} \cdot a^\xi(t_i,p_{t_i}^j,k_{t_i}^j)) - \theta  K(a^\xi(t_i,p_{t_i}^j,k_{t_i}^j))  \big) e^{-\rho t_i} \right]  \nonumber \\& + \mathbb{E}\left[g(k^j_T, p_T^j)e^{-\rho T}\right],
\end{align*}
where for $i=1, \ldots, N_T-1$,
\begin{align*}
p_{t_{i+1}}^j =&  p_{t_i}^j + \Phi( R^{j-1}_{t_i},p_{t_i}^j) \Delta t+ \gamma(p_{t_i}^j) (W^0_{t_{i+1}} -W^0_{t_i}), \\
k_{t_{i+1}}^j= & k_{t_i}^j + (a^{\xi}(t_i,p_{t_i}^j, k_{t_i}^j) - \delta k_{t_i}^j) \Delta t + \sigma(k_{t_i}^j) (W_{t_{i+1}} - W_{t_i}),
\end{align*}
and $p_0^j = \eta$, $k_0^j= \kappa$,
\item Then we estimate $R^j_t$. Introducing a second feedforward network 
\begin{align*}
    b^{\hat \xi} \colon [0,T] \times \R^d \to \R^d,
\end{align*}
with parameters $\hat \xi$, we solve
\begin{align*}
    \hat \xi^*_j = \argmax_{\hat \xi}  \sum_{i=1}^{N_T} \mathbb{E}[ \big(b^{\hat \xi}(t_i, p_{t_i}^j)- \phi( k_{t_i}^j) \big)^2],
\end{align*}
and we set 
\begin{align}
\label{eq:RUpdate}
    R^j_t = b^{\hat \xi^*_j}(t,p^j_t).
\end{align}
\end{enumerate}
The algorithm is stopped when 
$ \sum_{i=1}^{N_T} \mathbb{E}[(p^{j+1}_{t_i}-p^j_{t_i})^2] < \epsilon$.

In practice, in order to avoid oscillations during  iterations when $T$ is high, we use a fictitious version of the algorithm (see for example \cite{cardaliaguet2017learning}) and we replace
equation \eqref{eq:RUpdate} by 
\begin{align*}
    R^j_t = \frac{1}{j+1} \sum_{k=0}^j b^{\hat \xi^*_k}(t, p^k_t),
\end{align*}
with $b^{\hat \xi^*_0} =0$.\\
All distribution plots were obtained using 50 time steps and a neural network with 3 hidden layers and 20 neurons with a $\tanh$ activation function. During the first phase of the algorithm, where we estimate the optimal control, we use a batch size of 1000 samples of the common noise, and for each realization of the common noise we use a single realization of the idiosyncratic noise used for the investment to estimate $a^\xi$. In this first phase, we perform 2000 gradient descent iterations using the ADAM algorithm.

 In the second phase of the algorithm, in order to estimate $b^{\hat \xi}$, we must accurately approximate a conditional expectation. At each step of this second gradient descent algorithm, we use 100 samples of the common noise, and for each sample of the common noise we use 10,000 samples of the idiosyncratic noise used for the investment. This part is again solved using 2000 iterations of the ADAM algorithm.\\

 Therefore, we use two gradient descent procedures, initializing both ADAM optimizers with a learning rate of $1e{-3}$ at the beginning of the global resolution.

 \begin{remark}
     As the two resolutions are coupled through a fixed-point algorithm, there is no need to enforce full convergence of both phases by using an excessively large number of iterations.
 \end{remark}

 All results were obtained on an NVIDIA H100 GPU. With these parameters, the memory usage was 30 GB.\\

 The number of iterations of the fixed-point algorithm strongly depends on the model parameters. With the parameters used for the figures, we needed 49 iterations to reach a convergence threshold of $\epsilon = 1e{-3}$, and the total computation time was 111,000 seconds.

\subsection{Numerical results} \label{sec:num-result}

We have used the parameters in Table \ref{tab:simu_parameters} below for the simulation. We consider two sectors of activity ($n=2$): a brown one source of pollution and a green one. The utility function is a power function and the production function is a Constant Elasticity of Substitution (CES) production function minus the costs of production, i.e. we consider  
\begin{displaymath}
    F(k,p) = A\left(k_1^\gamma+k_2^\gamma\right)^{\beta/\gamma} -  b_1(p)k_1-  b_2k_2
\end{displaymath}
where $\beta $ belongs to $(0,1)$, $k_1$ represents the capital level in the brown sector and $k_2$ in the green sector. The coefficients $b_1(p)$ and $b_2$ are the production costs of each sector: we assume that the green sector is not affected by the environmental variable, while the brown sector is negatively affected by it. We consider a Dirac mass as the initial distribution of capital.

\begin{center}
    \begin{tabular}{|c|c|}
        \hline
        Parameters&Value\\
        \hline
        $\delta$ & $0.1$\\
        \hline
        $\sigma$ & $0.05$\\
        \hline
        $\sigma_0$ & $0.15$\\
        \hline
        $u(c)$ & $c^{0.8}$ \\ 
        \hline
        $\rho$ & $0.1$ \\
        \hline
        $F(k,p)$ & $A(k_1^{0.9} + k_2^{0.9})^\frac{1}{0.9} - b_1(p)k_1 - b_2 k_2$ \\
        \hline
        $b_1(p)$ &  $ 0.05 + (1 - e^{- 0.1p})$\\
        \hline $b_2$ & $0.35$\\
        \hline
        $g(k,p)$ & $u(F(k,p))e^{-\rho T}/\rho $\\
        \hline
        $\varepsilon(a)$ & $-0.1 \sum_{i=1}^2 a^i\ln(a^i)$\\
        \hline
        $m_0$ & $\delta_{(1,1)}$\\
        \hline
        $\gamma_0$ & $0.15$\\
        \hline
        $\Phi(e,p)$ & $0.3e-0.1p$\\
        \hline
        $\xi(k)$ & $0.5 k^1$\\
        \hline
    \end{tabular}
    \captionof{table}{\label{tab:simu_parameters}Model parameters}
\end{center}

Before providing a comment on the results, we will briefly introduce Figures \ref{fig:pollution}, \ref{fig:distributions_1}, and \ref{fig:pollution_quantiles}.
Figure \ref{fig:pollution} shows the evolution over time of the pollution level, the production costs $b_1(p)$ of the brown sector as well as the production costs $b_2$ of the green sector, the average production and consumption for two different realisations of the common noise. In what follows, blue curves refers to one realisation, and orange curves to the other.

\begin{center}
    \begin{tabular}{cc}
         \includegraphics[scale=0.4]{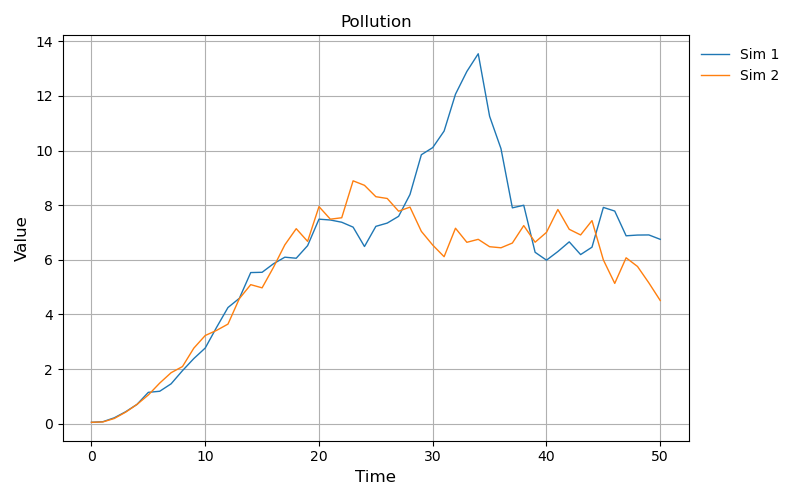}& 
         \includegraphics[scale=0.4]{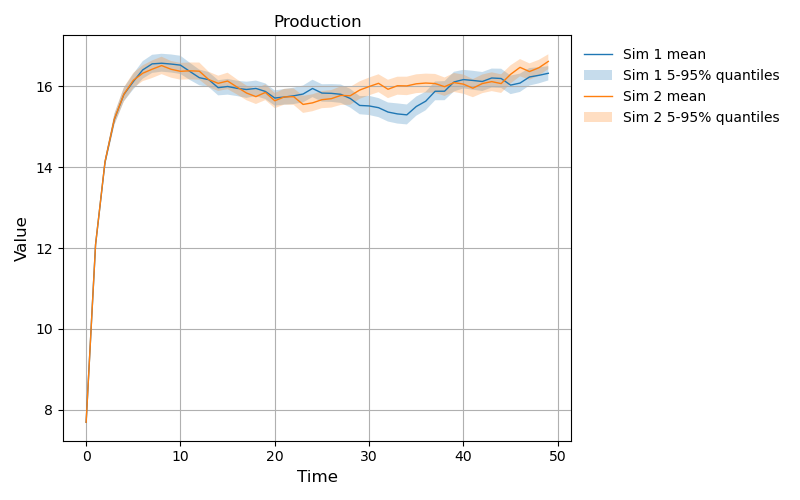}\\
         \includegraphics[scale=0.4]{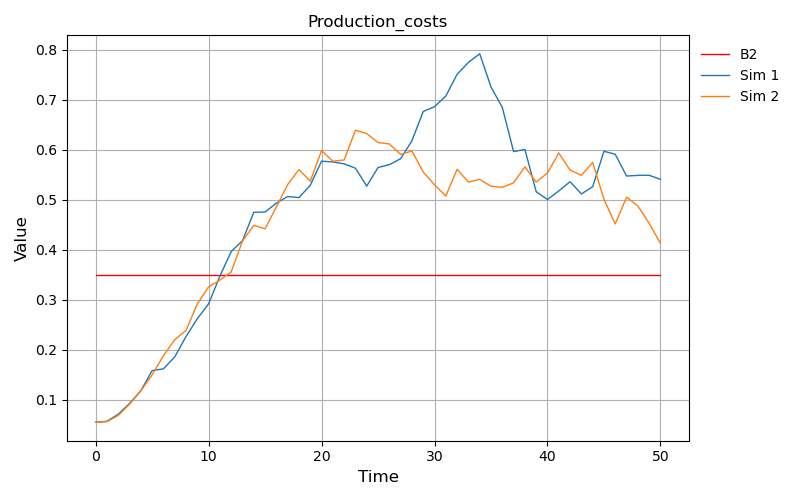} &
         \includegraphics[scale=0.4]{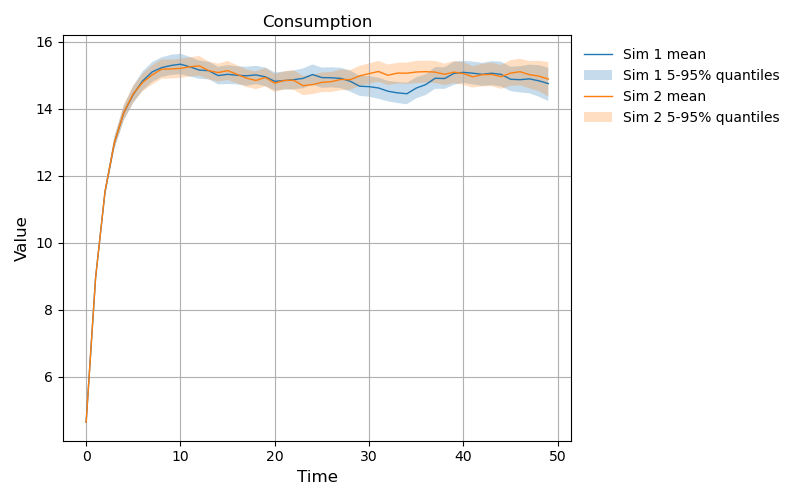}
    \end{tabular}
    \captionof{figure}{\label{fig:pollution}Two realisations of the common noise}.
\end{center}

We observe an increase of production, leading to a corresponding rise in consumption. This growth results in higher pollution levels, causing a increase of production costs in the brown sector $b_1(p)$.

 Figures \ref{fig:distributions_1} below, show the investment distributions (on the left) and the capital distributions (on the right) as a function of time. 
\begin{center}
        \begin{tabular}{cc}
        \includegraphics[scale=0.4]{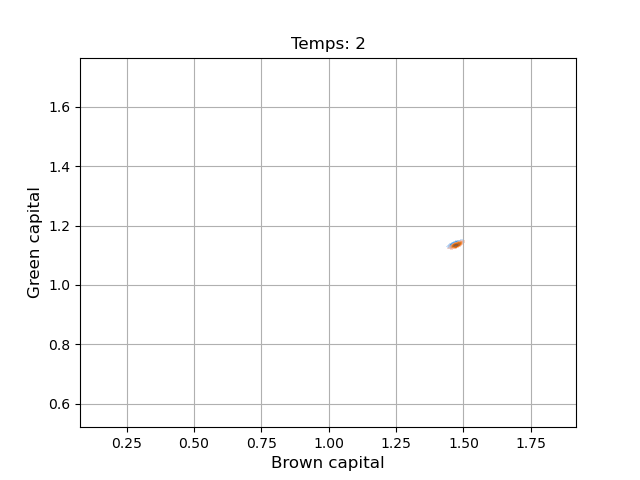} & \includegraphics[scale=0.4]{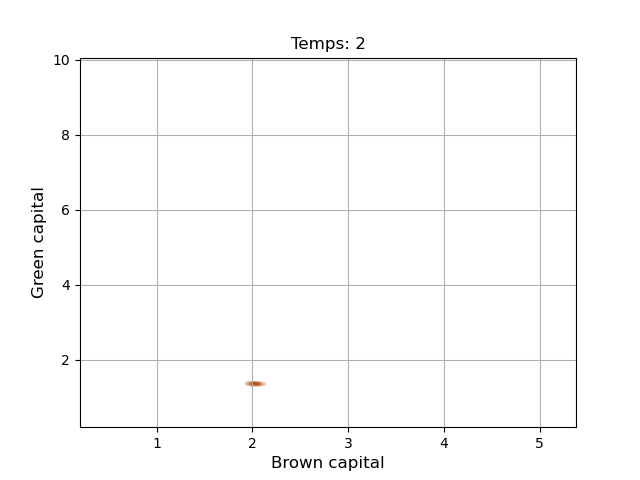}  \\
         \includegraphics[scale=0.4]{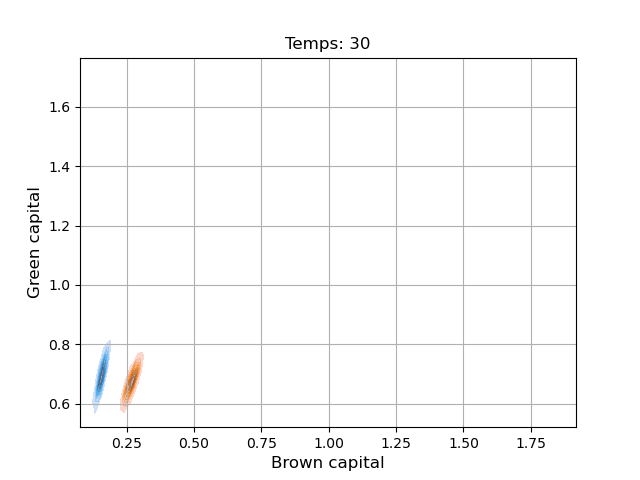} & \includegraphics[scale=0.4]{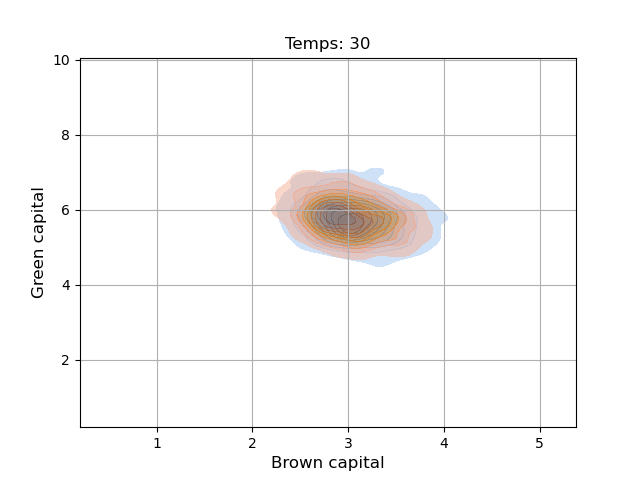}  \\
         \includegraphics[scale=0.4]{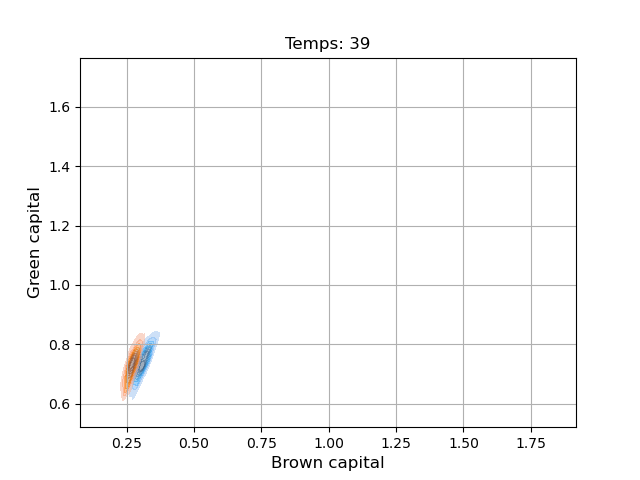} & \includegraphics[scale=0.4]{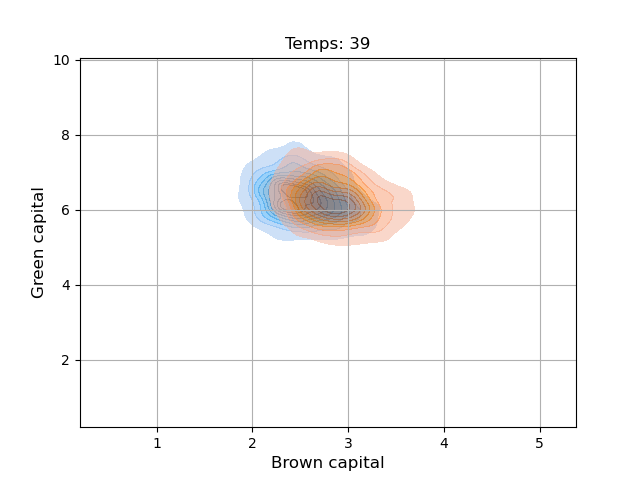}  \\
         \includegraphics[scale=0.4]{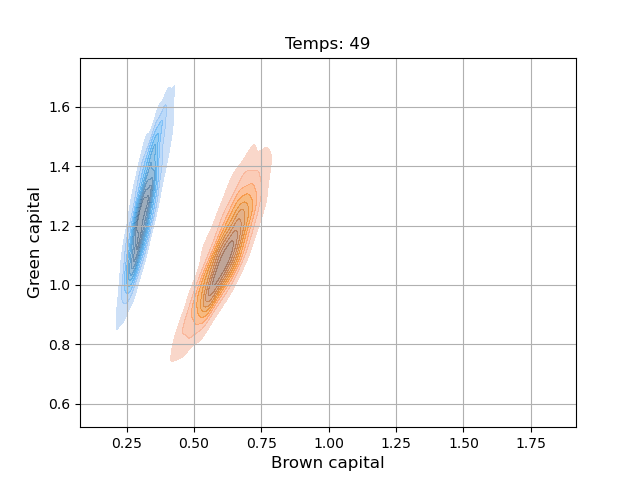} & \includegraphics[scale=0.4]{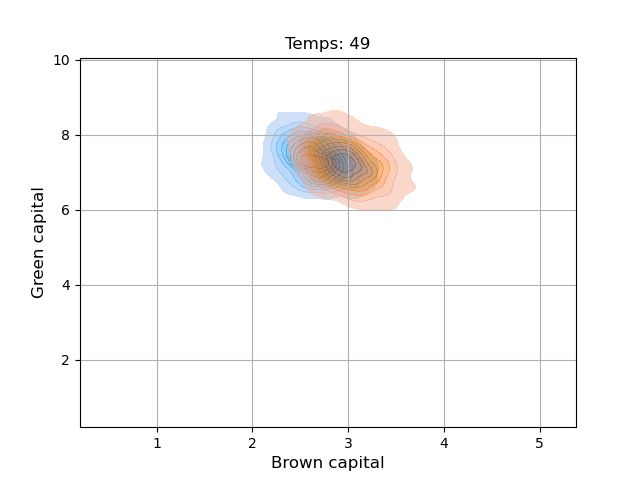}
    \end{tabular}
    \captionof{figure}{\label{fig:distributions_1}Distributions of the investment $a_t$  (on the left) and the capital $k_t$ (on the right) for two realisations of the common noise} at time $t=$2, 30, 39, and 49.
\end{center}

Figure \ref{fig:pollution_quantiles} illustrates the 5\%, 95\% quantiles, and the mean of the pollution process. The widening spread between the quantiles highlights the significant impact of common noise, which initially increases before stabilizing. Similarly, pollution levels rise and eventually stabilize.

\begin{center}
    \includegraphics[scale=0.6]{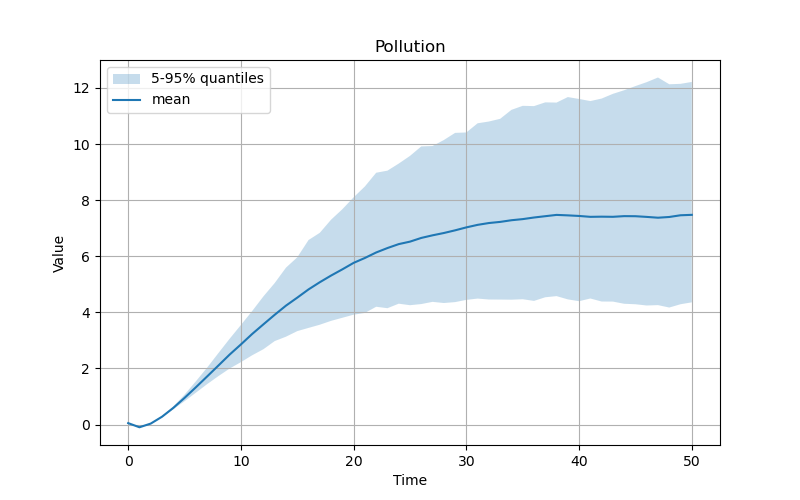}
    \captionof{figure}{\label{fig:pollution_quantiles}Quantiles 5\%, 95\% and mean of the pollution process }
\end{center}

\paragraph{Discussions on the outputs.}

Examining the pollution dynamics illustrated in Figure \ref{fig:pollution}, we observe a rapid increase in pollution in both scenarios, driven by the swift expansion of the brown sector. This trend can be attributed to the production costs: Figure \ref{fig:pollution} highlights the fact that, at first, the brown has smaller production costs, $b_1(p)$, than the green sector, $b_2$ (represented by the red line in the figure). Consequently, countries are initially incentivized to invest in the brown sector, as depicted in Figure \ref{fig:distributions_1}. 

Note that at time t=2, and even if the Brown sector is much more attractive, countries invest in Green capital. It reflects countries’ expectations of increasing pollution, which would eventually render the brown sector less competitive. Moreover, the sector is more and more exposed to common noise which also degrade the agents utility since there are risk-averse. As a result, it becomes progressively advantageous to develop the green sector, marking already the beginning of a transition. It is relevant for real-life applications such as the transition of heater systems. If a country has a significant part of heater-systems based on gas, its two main drivers to transition toward cleaner energies such as local biomass would be its expectations for the price of the gas to rise and to fluctuate significantly.

Returning to the pollution process, its volatility grows over time. This is linked to the rising pollution levels and the scaling of the volatility rate with these levels. By time $t=30$, random shocks result in the pollution trajectories of the two scenarios diverging. Consequently, the optimal investment strategies also differ. Figures \ref{fig:distributions_1} reveals that the scenario with higher pollution levels (the blue scenario) allocates more investment to the green sector. These differing policies lead to diverging capital distributions: by time $t=39$. Figure \ref{fig:pollution_quantiles} further supports that there is a stabilisation effect due to the drift in equation \eqref{eq:dynamics_p} since the pollution process appears to converge toward a steady state. Indeed, the stabilising effect comes from the 2 components of the external variable drift: The higher the pollution, the smaller the linear term. Moreover, the higher the pollution, the higher the incentives of the agents to develop the green sector instead of the brown sector, leading to a contribution to pollution less important. This two effects contribute to decrease the pollution level. Of course the opposite reasoning holds.

Finally, we note that the random shocks influencing the pollution process in the blue scenario have had a negative impact on production and consumption levels compared to the orange scenario.

\subsection{Sensitivity of the results to the resolution parameters}
The resolution depends on the neural network parameters, namely the number of hidden layers, the number of neurons, the activation function, and the time-step discretization of the Euler scheme. The activation functions tested were the classical $\tanh$, ELU, and SiLU functions. The classical ReLU activation led to instabilities (“explosions”) in the resolution when using 3 layers.

In Figure \ref{fig:traj-1}, we plot pollution trajectories for the same realizations of the Brownian motions, using 50 time steps. The type of activation function has no visible effect on the solution. The results are stable with respect to the number of neurons, but we observe a slight variation when increasing the number of layers.

\begin{center}
    \centering
    \begin{minipage}[t]{0.49\linewidth}
     \includegraphics[width=\linewidth]{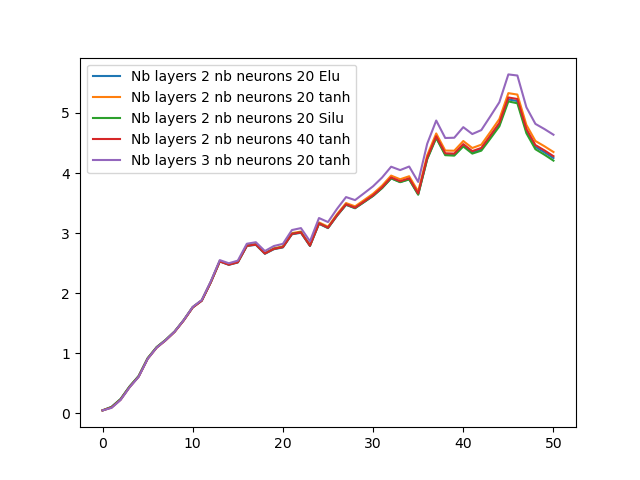}
    \centering
    Trajectory 1
    \end{minipage}
    \begin{minipage}[t]{0.49\linewidth}
    \includegraphics[width=\linewidth]{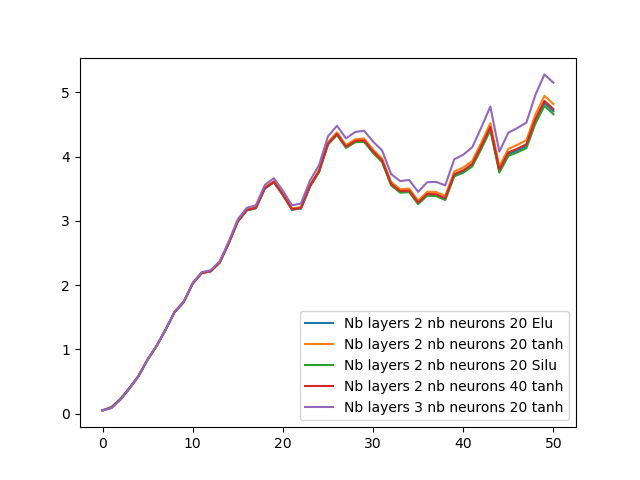}
    \centering
    Trajectory 2   
    \end{minipage}
    \begin{minipage}[t]{0.49\linewidth}
    \includegraphics[width=\linewidth]{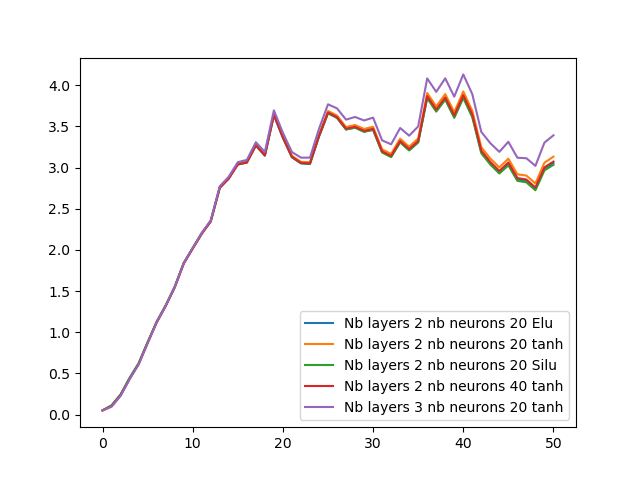}
    \centering
    Trajectory 3
    \end{minipage}
     \begin{minipage}[t]{0.49\linewidth}
    \includegraphics[width=\linewidth]{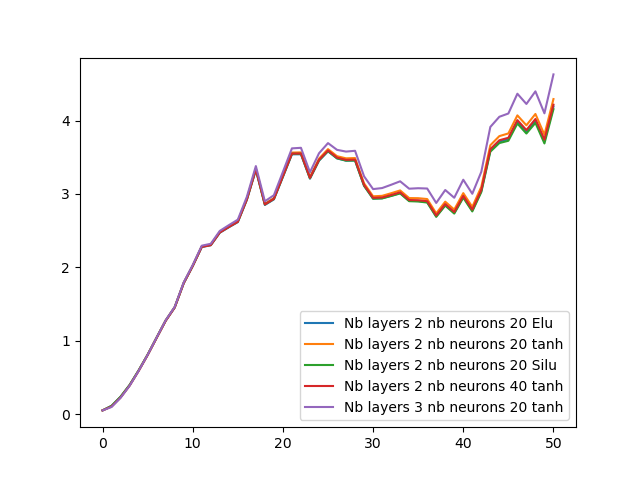}
    \centering
    Trajectory 4   
    \end{minipage}
    \captionof{figure}{\label{fig:traj-1} Sensitivity of pollution trajectories with respect to the neural network parameters}
\end{center}

Finally, on Figure \ref{fig:traj-2}, for the same Brownian trajectories, we display pollution trajectories obtained with different numbers of time steps. Again, the solution is very stable with respect to the time discretization.

\begin{center}
    \centering
    \begin{minipage}[t]{0.49\linewidth}
     \includegraphics[width=\linewidth]{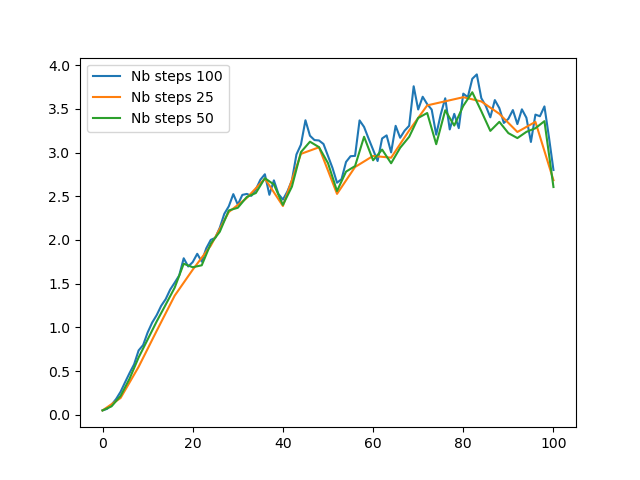}
    \centering
    Trajectory 1
    \end{minipage}
    \begin{minipage}[t]{0.49\linewidth}
    \includegraphics[width=\linewidth]{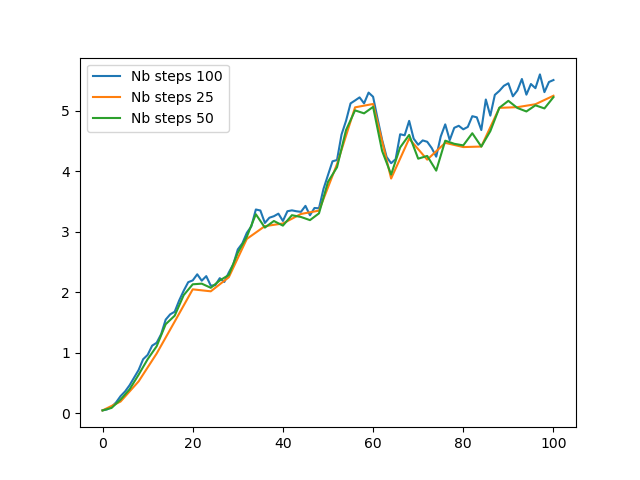}
    \centering
    Trajectory 2    
    \end{minipage}
    \begin{minipage}[t]{0.49\linewidth}
    \includegraphics[width=\linewidth]{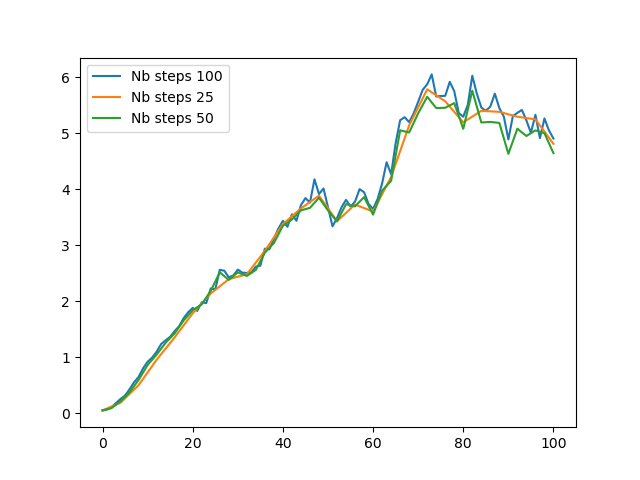}
    \centering
    Trajectory 3   
    \end{minipage}
     \begin{minipage}[t]{0.49\linewidth}
    \includegraphics[width=\linewidth]{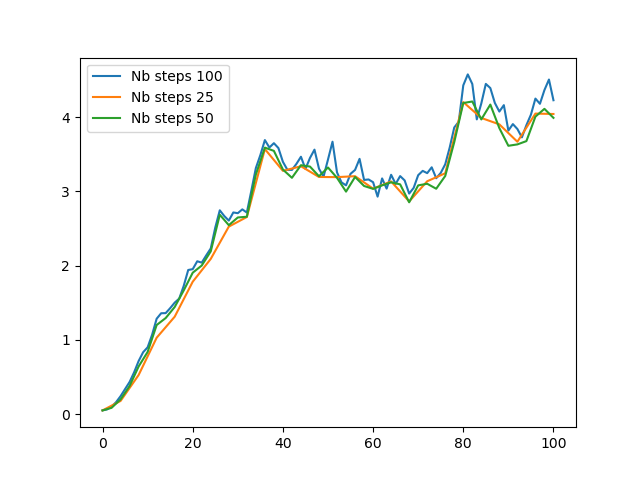}
    \centering
    Trajectory 4    
    \end{minipage}
    \captionof{figure}{\label{fig:traj-2} Sensitivity of pollution trajectories to the number of time steps using a network with 3 hidden layers, 20 neurons, and $\tanh$ activation.}
\end{center}

\paragraph{Acknowledgments.}
We thank Raouf Boucekkine, Charles Bertucci, François Delarue, Fabrice Mao Djete, Hugo Lecl{\`e}re, Charles Meynard and Anastassios Xepapadeas for helpful discussions and comments. 

\smallskip

Quentin Petit and Xavier Warin are supported by FiME, Laboratoire de Finance des March{\'e}s de l’Energie, and the ”Finance and Sustainable Development” EDF.

Pierre Lavigne acknowledges the financial support of the European Research Council (ERC) under the European Union’s Horizon Europe research and innovation programme (AdG ELISA project, Grant agreement No. 101054746). Views and opinions expressed are however those of the authors only and do not necessarily reflect those of the European Union or the European Research Council Executive Agency. Neither the European Union nor the granting authority can be held responsible for them.

\paragraph{Data availability statement:} No data were used in this study.

\bibliographystyle{abbrv}
\bibliography{biblio}

@book{acemoglu2008introduction,
  title={Introduction to modern economic growth},
  author={Acemoglu, Daron},
  year={2008},
  publisher={Princeton university press}
}

@article{bilal2025macroeconomics,
  title={Macroeconomics and climate change},
  author={Bilal, Adrien and Stock, James H},
  year={2025},
  journal={},
  publisher={National Bureau of Economic Research}
}

@article{nordhaus1993rolling,
  title={Rolling the ‘DICE’: an optimal transition path for controlling greenhouse gases},
  author={Nordhaus, William D},
  journal={Resource and Energy Economics},
  volume={15},
  number={1},
  pages={27--50},
  year={1993},
  publisher={Elsevier}
}

@article{nordhaus1996regional,
  title={A regional dynamic general-equilibrium model of alternative climate-change strategies},
  author={Nordhaus, William D and Yang, Zili},
  journal={The American Economic Review},
  pages={741--765},
  year={1996},
  publisher={JSTOR}
}

@techreport{aid2025regulation,
  title={Regulation in a mean-field investment game with climate damage},
  author={A{\"\i}d, Ren{\'e} and Federico, Salvatore and Ferrari, Giorgio and Rodosthenous, Neofytos},
  year={2025},
  institution={Center for Mathematical Economics Working Papers}
}

@article{bertucci2024noise,
  title={Noise through an additional variable for mean field games master equation on finite state space},
  author={Bertucci, Charles and Meynard, Charles},
  journal={arXiv preprint arXiv:2402.05635},
  year={2024}
}

@article{meynard2024study,
  title={A study of common noise in mean field games},
  author={Bertucci, Charles and Meynard, Charles},
  journal={arXiv preprint arXiv:2412.12741},
  year={2024}
}

@article{hardin1968tragedy,
  title={The tragedy of the commons: the population problem has no technical solution; it requires a fundamental extension in morality.},
  author={Hardin, Garrett},
  journal={science},
  volume={162},
  number={3859},
  pages={1243--1248},
  year={1968},
  publisher={American Association for the Advancement of Science}
}

@article{Gangbo2022a,
author = {Wilfrid Gangbo and Alp{\'a}r R. M{\'e}sz{\'a}ros and Chenchen Mou and Jianfeng Zhang},
title = {{Mean field games master equations with nonseparable Hamiltonians and displacement monotonicity}},
volume = {50},
journal = {The Annals of Probability},
number = {6},
publisher = {Institute of Mathematical Statistics},
pages = {2178 -- 2217},
year = {2022},
doi = {10.1214/22-AOP1580},
URL = {https://doi.org/10.1214/22-AOP1580}
}

@misc{Delsarto2024,
      title={A Mean Field Game approach for pollution regulation of competitive firms}, 
      author={Gianmarco Del Sarto and Marta Leocata and Giulia Livieri},
      year={2024},
      eprint={2407.12754},
      archivePrefix={arXiv},
      primaryClass={q-fin.MF},
      url={https://arxiv.org/abs/2407.12754}, 
}

@article{Gangbo2022b,
author = {Gangbo, Wilfrid and Mészáros, Alpár R.},
title = {Global Well-Posedness of Master Equations for Deterministic Displacement Convex Potential Mean Field Games},
journal = {Communications on Pure and Applied Mathematics},
volume = {75},
number = {12},
pages = {2685-2801},
doi = {https://doi.org/10.1002/cpa.22069},
url = {https://onlinelibrary.wiley.com/doi/abs/10.1002/cpa.22069},
eprint = {https://onlinelibrary.wiley.com/doi/pdf/10.1002/cpa.22069},
year = {2022}
}

@article{Meszaros2024,
author = {M\'{e}sz\'{a}ros, Alp\'{a}r R. and Mou, Chenchen},
title = {Mean Field Games Systems under Displacement Monotonicity},
journal = {SIAM Journal on Mathematical Analysis},
volume = {56},
number = {1},
pages = {529-553},
year = {2024},
doi = {10.1137/22M1534353},

URL = { 
    
        https://doi.org/10.1137/22M1534353
    
    

},
eprint = { 
    
        https://doi.org/10.1137/22M1534353
    
    

}
}

@article{weyant2017some,
  title={Some contributions of integrated assessment models of global climate change},
  author={Weyant, John},
  journal={Review of Environmental Economics and Policy},
  year={2017},
  publisher={The University of Chicago Press}
}

@article{huang2006large,
author = {Minyi Huang and Roland P. Malham{\'e} and Peter E. Caines},
title = {{Large population stochastic dynamic games: closed-loop McKean-Vlasov systems and the Nash certainty equivalence principle}},
volume = {6},
journal = {Communications in Information \& Systems},
number = {3},
publisher = {International Press of Boston},
pages = {221 -- 252},
year = {2006},
}

@article{achdou2023mean,
  title={A mean field model for the interactions between firms on the markets of their inputs},
  author={Achdou, Yves and Carlier, Guillaume and Petit, Quentin and Tonon, Daniela},
  journal={Mathematics and Financial Economics},
  pages={1--35},
  year={2023},
  publisher={Springer}
}

@book{carmona2018probabilistic,
  title={Probabilistic theory of mean field games with applications I-II},
  author={Carmona, Ren{\'e} and Delarue, Fran{\c{c}}ois and others},
  year={2018},
  publisher={Springer}
}

@article{lasry2007mean,
  title={Mean field games},
  author={Lasry, Jean-Michel and Lions, Pierre-Louis},
  journal={Japanese journal of mathematics},
  volume={2},
  number={1},
  pages={229--260},
  year={2007},
  publisher={Springer}
}

@article{tahvonen1993economic,
  title={Economic growth, pollution, and renewable resources},
  author={Tahvonen, Olli and Kuuluvainen, Jari},
  journal={Journal of Environmental Economics and Management},
  volume={24},
  number={2},
  pages={101--118},
  year={1993},
  publisher={Elsevier}
}

@article{djete2023large,
  title={Large population games with interactions through controls and common noise: convergence results and equivalence between open-loop and closed-loop controls},
  author={Djete, Mao Fabrice},
  journal={ESAIM: Control, Optimisation and Calculus of Variations},
  volume={29},
  pages={39},
  year={2023},
  publisher={EDP Sciences}
}

@article{carmona2016mean,
author = {Ren{\'e} Carmona and Fran{\c{c}}ois Delarue and Daniel Lacker},
title = {{Mean field games with common noise}},
volume = {44},
journal = {The Annals of Probability},
number = {6},
publisher = {Institute of Mathematical Statistics},
pages = {3740 -- 3803},
year = {2016},
doi = {10.1214/15-AOP1060},
URL = {https://doi.org/10.1214/15-AOP1060}
}

@article{germain2022numerical,
  title={Numerical resolution of McKean-Vlasov FBSDEs using neural networks},
  author={Germain, Maximilien and Mikael, Joseph and Warin, Xavier},
  journal={Methodology and Computing in Applied Probability},
  volume={24},
  number={4},
  pages={2557--2586},
  year={2022},
  publisher={Springer}
}

@article{ahuja2016wellposedness,
  title={Wellposedness of mean field games with common noise under a weak monotonicity condition},
  author={Ahuja, Saran},
  journal={SIAM Journal on Control and Optimization},
  volume={54},
  number={1},
  pages={30--48},
  year={2016},
  publisher={SIAM}
}

@book{cardaliaguet2019master,
  title={The master equation and the convergence problem in mean field games:(ams-201)},
  author={Cardaliaguet, Pierre and Delarue, Fran{\c{c}}ois and Lasry, Jean-Michel and Lions, Pierre-Louis},
  year={2019},
  publisher={Princeton University Press}
}

@article{kobeissi2024tragedy,
  title={The tragedy of the commons: A Mean-Field Game approach to the reversal of travelling waves},
  author={Kobeissi, Ziad and Mazari-Fouquer, Idriss and Ruiz-Balet, Dom{\`e}nec},
  journal={Nonlinearity},
  volume={37},
  number={11},
  pages={115010},
  year={2024},
  publisher={IOP Publishing}
}

@article{lavigne2023decarbonization,
  title={Decarbonization of financial markets: a mean-field game approach},
  author={Lavigne, Pierre and Tankov, Peter},
  journal={arXiv preprint arXiv:2301.09163},
  year={2023}
}

@article{brock2010green,
  title={The green Solow model},
  author={Brock, William A and Taylor, M Scott},
  journal={Journal of Economic Growth},
  volume={15},
  pages={127--153},
  year={2010},
  publisher={Springer}
}

@article{smulders2014growth,
  title={Growth theory and ‘green growth’},
  author={Smulders, Sjak and Toman, Michael and Withagen, Cees},
  journal={Oxford review of economic policy},
  volume={30},
  number={3},
  pages={423--446},
  year={2014},
  publisher={Oxford University Press UK}
}

@article{uzawa1961two,
  title={On a two-sector model of economic growth},
  author={Uzawa, Hirofumi},
  journal={The Review of Economic Studies},
  volume={29},
  number={1},
  pages={40--47},
  year={1961},
  publisher={Wiley-Blackwell}
}

@article{zhou2024best,
  title={On a best response problem arising in mean field stochastic growth games with common noise},
  author={Zhou, Mengjie and Huang, Minyi},
  journal={Asian Journal of Control},
  year={2024},
  publisher={Wiley Online Library}
}

@article{achdou2022income,
  title={Income and wealth distribution in macroeconomics: A continuous-time approach},
  author={Achdou, Yves and Han, Jiequn and Lasry, Jean-Michel and Lions, Pierre-Louis and Moll, Benjamin},
  journal={The review of economic studies},
  volume={89},
  number={1},
  pages={45--86},
  year={2022},
  publisher={Oxford University Press}
}

@article{bertucci2023monotone,
  title={Monotone solutions for mean field games master equations: continuous state space and common noise},
  author={Bertucci, Charles},
  journal={Communications in Partial Differential Equations},
  volume={48},
  number={10-12},
  pages={1245--1285},
  year={2023},
  publisher={Taylor \& Francis}
}

@inproceedings{gomes2016mean,
  title={A mean-field game economic growth model},
  author={Gomes, Diogo and Lafleche, Laurent and Nurbekyan, Levon},
  booktitle={2016 American Control Conference (ACC)},
  pages={4693--4698},
  year={2016},
  organization={IEEE}
}

@book{pham2009continuous,
  title={Continuous-time stochastic control and optimization with financial applications},
  author={Pham, Huy{\^e}n},
  volume={61},
  year={2009},
  publisher={Springer Science \& Business Media}
}

@book{zhang2017backward,
  title={Backward stochastic differential equations},
  author={Zhang, Jianfeng},
  year={2017},
  publisher={Springer}
}

@article{carmona2022convergence,
  title={Convergence analysis of machine learning algorithms for the numerical solution of mean field control and games: II—the finite horizon case},
  author={Carmona, Ren{\'e} and Lauri{\`e}re, Mathieu},
  journal={The Annals of Applied Probability},
  volume={32},
  number={6},
  pages={4065--4105},
  year={2022},
  publisher={Institute of Mathematical Statistics}
}

@article{chan2019machine,
  title={Machine learning for semi linear PDEs},
  author={Chan-Wai-Nam, Quentin and Mikael, Joseph and Warin, Xavier},
  journal={Journal of scientific computing},
  volume={79},
  number={3},
  pages={1667--1712},
  year={2019},
  publisher={Springer}
}

@article{cardaliaguet2017learning,
  title={Learning in mean field games: the fictitious play},
  author={Cardaliaguet, Pierre and Hadikhanloo, Saeed},
  journal={ESAIM: Control, Optimisation and Calculus of Variations},
  volume={23},
  number={2},
  pages={569--591},
  year={2017},
  publisher={EDP Sciences}
}

@misc{tangpi2025,
      title={Mean field games with common noise via Malliavin calculus}, 
      author={Ludovic Tangpi and Shichun Wang},
      year={2025},
      eprint={2405.02244},
      archivePrefix={arXiv},
      primaryClass={math.PR},
      url={https://arxiv.org/abs/2405.02244}, 
}

@article{Cardaliaguet2018,
  title={Mean field game of controls and an application to trade crowding},
  author={Cardaliaguet, Pierre and Lehalle, Charles-Albert},
  journal={Mathematics and Financial Economics},
  volume={12},
  pages={335--363},
  year={2018},
  publisher={Springer}
}

\end{document}